 \documentclass[12pt, reqno]{amsart}
 \usepackage{amsmath, amsthm, amscd, amsfonts, amssymb, graphicx, color}
 \usepackage[bookmarksnumbered, colorlinks, plainpages]{hyperref}

 \newtheorem{theorem}{Theorem}[section]
 \newtheorem{lemma}[theorem]{Lemma}
 \newtheorem{proposition}[theorem]{Proposition}
 \newtheorem{corollary}[theorem]{Corollary}
 \theoremstyle{definition}
 \newtheorem{definition}[theorem]{Definition}
 \newtheorem{example}[theorem]{Example}

 \theoremstyle{remark}
 \newtheorem{remark}[theorem]{Remark}
 \numberwithin{equation}{section}

\makeatletter
\@namedef{subjclassname@2020}{%
  \textup{2020}Mathematics Subject Classification}
\makeatother

\begin{document}

    \setcounter{page}{1}
\title[The Automorphism Group of the Spectral Incidence Graph over F. F.]
{The Automorphism Group of the Spectral Incidence Graph over Finite Fields}

\author[ A. Majidinya ]{ Ali Majidinya }
\address{Department of Mathematics and Computer Science, Faculty of Basic Sciences,
 Salman Farsi University of Kazerun, Kazerun, Iran, P.O. Box 73175-457.}
\email{\textcolor[rgb]{0.00,0.00,0.84}{ali.majidinya@gmail.com and
ali.majidinya@kazerunsfu.ac.ir}}

\begin{abstract}
Let $\mathbb{F}_q$ be a finite field with $q$ elements and let
  $\mathbb{V}_0=\mathbb{F}_q^n$ be the $n$-dimensional
vector space over $\mathbb{F}_q$, for an integer
$n\geq2$. We introduce the \emph{spectral
incidence graph} of $\mathbb{V}_0$, denoted by
$\mathbf{SIG}(\mathbb{V}_0)$, a bipartite graph whose two vertex classes
consist of the one-dimensional subspaces of $M_n(\mathbb{F}_q)$
generated by matrices having at least one eigenvector in $\mathbb{V}_0$,
and the one-dimensional subspaces of $\mathbb{V}_0$ (i.e.  the
projective space $\mathbb{P}^{n-1} (\mathbb{F}_q)$), respectively. A matrix vertex
$\langle M\rangle$ and a point vertex $\langle v\rangle$ are adjacent
if and only if $v$ is an eigenvector of $M$. We determine several
structural parameters of $\mathbf{SIG}(\mathbb{V}_0)$, including its twin classes,
connectivity, domination number, diameter, vertex
degrees, and number of edges. We show that every
automorphism of $\mathbf{SIG}(\mathbb{V}_0)$ preserves the two parts of the bipartition.
For $n\geq 3$,  the restriction of every graph automorphism to the part $\mathbb{P}^{n-1} (\mathbb{F}_q)$
 induces a semilinear collineation of
$\mathbb{P}^{n-1} (\mathbb{F}_q)$.
We prove that the kernel of the induced action of
$Aut(\mathbf{SIG}(\mathbb V_0))$ on the
projective part $\mathbb{P}^{n-1} (\mathbb{F}_q)$
 is precisely $\prod_{\mathcal C\in\mathcal T} S_{\mathcal C}$, where $\mathcal T$ denotes
  the set of twin classes of matrix vertices and
 $S_{\mathcal{C}}$ is the symmetric group on $\mathcal{C}$.
 We then determine the full automorphism
group of $\mathbf{SIG}(\mathbb{V}_0)$. For $n\geq3$, we prove that
$Aut(\mathbf{SIG}(\mathbb V_0))
\cong
\left(\prod_{\mathcal C\in\mathcal T} S_{\mathcal C}\right)
\rtimes P\Gamma L_n(q)$.
  For $n=2$, we obtain
$Aut(\mathbf{SIG}(\mathbb V_0))
\cong
\left(
\prod_{i=1}^{q+1}S_q\times \prod_{i=1}^{\frac{q(q+1)}{2}}S_q
\right)
\rtimes S_{q+1}.$
\end{abstract}
\maketitle
\section*{\bf Introduction}%
Graphs constructed from algebraic and geometric structures provide a
natural framework for translating algebraic properties into combinatorial ones.
In particular, graphs associated with finite vector spaces and finite geometries have
been studied extensively,
with emphasis on their structural properties and automorphism groups. For graphs related to
vector spaces; see
\cite{Das0}, \cite{Ma1}, \cite{Majid} and \cite{wong2}. There are also studies of
graphs associated with symplectic, orthogonal, and unitary spaces over
finite fields; see \cite{Gu1} and \cite{Gu2}. The automorphism groups of
such graphs are of particular interest, and several results in this
direction can be found in
\cite{Gu3}, \cite{CMa}, \cite{Majid}, \cite{SMayo}, \cite{XWang1},
\cite{LWang1}, \cite{LWang2} and \cite{wang1}. In particular, for
automorphism groups of graphs associated with vector spaces, we refer to
\cite{Majid}, \cite{XWang1} and the references therein. Automorphism and
 isometric-embedding problems for graphs arising from finite geometric structures have
 also been studied extensively; see for example \cite{KwiatkowskiPankov2017}.\par
 Wong et al. in \cite{wong1} characterized the automorphisms of the
zero-divisor graph whose vertex set consists of all rank-one upper
triangular matrices over a finite field. In \cite{XWang1}, Wang et al.
studied the transformation graph of a vector space over a finite field.
They considered the problem of whether a linear transformation sends a
vector to the zero vector and interpreted this problem in graph-theoretic
terms. They investigated several structural parameters of the resulting
graph, including its diameter and domination number and also studied
its automorphism group.

Following \cite{XWang1}, Wang in \cite{wang1} introduced the dual graph of a
vector space $\mathbb{V}_0$ over a finite field $\mathbb{F}_q$, denoted by
$DG(\mathbb{V})$. This is a bipartite graph with vertex set
$V=X\cup X^*$, where $X$ is the set of one-dimensional subspaces of
$\mathbb{V}_0$ and $X^*$ is the set of one-dimensional subspaces of the
dual space of $\mathbb{V}_0$. A vertex $S\in X$ is adjacent to a vertex
$T\in X^*$ if and only if $f(s)=0$ for all $f\in T$ and all $s\in S$.
Wang determined several structural parameters of $DG(\mathbb{V})$,
including its domination number, independence number, diameter, and
girth, and proved that the graph is distance-transitive. He also
described several automorphisms of the dual graph.

Let $\mathbb{V}_0$ be an $n$-dimensional vector space over a finite field
$\mathbb{F}_q$, and let $\mathbb{T}_0$ denote the set of all linear
functionals on $\mathbb{V}_0$. Put
$\mathbb{V}=\mathbb{V}_0\setminus\{0\}$ and
$\mathbb{T}=\mathbb{T}_0\setminus\{0\}$. Following
\cite[page 13, part (c)]{XWang1}, in \cite{Majid} we introduced the
\emph{linear functional graph of a vector space}. This is a bipartite
graph with vertex set
$V=\mathbb{V}\cup\mathbb{T}$, where $f\in\mathbb{T}$ and
$v\in\mathbb{V}$ are adjacent if and only if $f(v)=0$. In that work,
we characterized the structure of all automorphisms of this
graph and determined the cardinality of its automorphism group.

A common objective in the works cited above is to determine the
automorphism groups of algebraically defined graphs in terms of
well-known and explicitly identifiable algebraic or geometric groups.
In several of these works, particular families of automorphisms or
distinguished subgroups of the full automorphism group are described.
In our previous work \cite{Majid}, we determined the structure of all
 automorphisms of the linear functional graph.
  The present paper continues this line of investigation for a graph whose incidence
relation is governed by eigenvectors of matrices rather than by the
vanishing of linear functionals.

Eigenvalue and eigenvector theory provides a fundamental framework for
understanding the structure of linear transformations and their
associated algebraic objects. From the viewpoint of projective geometry, eigenvectors
determine invariant projective points, or eigendirections, while
eigenspaces induce invariant projective subspaces under the associated
projective action. Over finite fields, these connections between linear
transformations, projective geometry, and spectral properties provide a
natural setting for a graph construction based on eigenvector incidence.

Motivated by these connections and by our previous work \cite{Majid}, we
introduce the \emph{Spectral Incidence Graph} associated with a
finite-dimensional vector space $\mathbb{V}_0$ over $\mathbb{F}_q$,
denoted by $\mathbf{SIG}(\mathbb{V}_0)$. Unlike incidence graphs arising
from vectors and linear functionals, this construction records the
incidence between projective classes of matrices and their invariant
one-dimensional subspaces, thereby incorporating spectral information into a
finite-geometric graph. The main objective of this paper is to determine
the full automorphism group of $\mathbf{SIG}(\mathbb{V}_0)$ and to
investigate several of its structural properties.

Using a combinatorial method we prove that every automorphism of
$\mathbf{SIG}(\mathbb V_0)$ preserves the two parts of its
bipartition; see Lemma~\ref{part to part}, and also for $n\geq3$, in Theorem \ref{collinearity} we show that
 the restriction of every graph automorphism to the
 projective-point part is a semilinear collineation of
 $\mathbb P^{n-1}(\mathbb F_q)$. By the Fundamental Theorem of Projective
Geometry, together with Lemma~\ref{corn3} and
Remark~\ref{formproject}, we show that the image
of the induced action of $Aut(\mathbf{SIG}(\mathbb V_0))$ on the part $\mathbb P^{n-1}(\mathbb F_q)$
 is precisely $P\Gamma L_n(q)$.
 In Theorems \ref{TwinNontrivial} and \ref{twinpoints} we characterize the nontrivial twin classes of $\mathbf{SIG}(\mathbb{V}_0)$
 and we show that in the part $\mathbb{P}^{n-1}(\mathbb{F}_q)$ we have no nontrivial twin class.
 Moreover, we prove that
the kernel of the induced action of $Aut(\mathbf{SIG}(\mathbb V_0))$ on the projective part
$\mathbb P^{n-1}(\mathbb F_q)$ is precisely
$
\prod_{\mathcal C\in\mathcal T}S_{\mathcal C},
$
where $\mathcal T$ is the set of twin classes of matrix vertices and $S_\mathcal{C}$ is the symmetric
group on $\mathcal C$; see Theorem \ref{kernel} and Lemma \ref{S_{q+1}}.
The resulting short exact sequence splits, yielding,  in Theorem \ref{Autn-geq3} for $n\geq3$,
\[
Aut(\mathbf{SIG}(\mathbb V_0))
\cong
\left(\prod_{\mathcal C\in\mathcal T}S_{\mathcal C}\right)
\rtimes P\Gamma L_n(q).
\]
The case $n=2$ is exceptional and is treated separately; in this case,
the image of the induced action of $Aut(\mathbf{SIG}(\mathbb V_0))$ on $\mathbb P^1(\mathbb F_q)$ is
the permutation group $S_{\mathbb P^1(\mathbb F_q)}$ which is isomorphic with  the full symmetric
group $S_{q+1}$ ;
 see Lemma \ref{S_{q+1}}. In Theorem \ref{Autn=2} for $n=2$ we obtain
\[
Aut(\mathbf{SIG}(\mathbb V_0))
\cong
\left(
\prod_{i=1}^{q+1}S_q
\times
\prod_{i=1}^{q(q+1)/2}S_q
\right)
\rtimes S_{q+1}.
\]
The  remainder of the paper is organized as follows.\par
 In Section~\ref{section1}, we establish preliminary results on matrix
theory and projective geometry and introduce the Spectral Incidence
Graph $\mathbf{SIG}(\mathbb{V}_0)$ for
$\mathbb{V}_0=\mathbb{F}_q^n$. We also investigate several basic
structural properties of this graph, including its connectivity,
diameter, domination number, vertex degrees and number of edges.\par
In Section~\ref{section2}, we study the twin classes of $\mathbf{SIG}(\mathbb{V}_0)$.
We characterize the nontrivial twin classes of $\mathbf{SIG}(\mathbb{V}_0)$
 and we show that in the part $\mathbb{P}^{n-1}(\mathbb{F}_q)$ we have no nontrivial twin class.
These results provide the structural
information needed for the determination of the automorphisms of $\mathbf{SIG}(\mathbb{V}_0)$.\par
Finally, in Section \ref{section3}, we determine the full automorphism
group of $\mathbf{SIG}(\mathbb{V}_0)$. For $n\geq3$, we use algebraic
methods together with the Fundamental Theorem of Projective Geometry
to determine the image of the induced action of $Aut(\mathbf{SIG}(\mathbb{V}_0))$ on the
 projective part $\mathbb{P}^{n-1}(\mathbb{F}_q)$ and obtain the
full automorphism group. The automorphism group for the case $n=2$ is exceptional and is treated
separately.
\section{\bf Preliminaries}\label{section1}
Let $\mathcal{G}=(V,E)$ be a simple graph with the vertex set $V$ and
edge set $E$. It is written $u\sim v$ if $u$ and $v$ are adjacent vertices
in the graph. Then
 $\deg(v)$ stands for the  \it degree \rm of $v\in V$, i.e.
the cardinality of the set of all vertices  which are adjacent to
$v$. A graph $\mathcal{G}=(V, E)$ is said to be a\emph{
bipartite graph} with vertex bipartition $V=X\cup Y$ if every edge
of the graph has one end in $X$ and another end in $Y$. A subset $D$
of the vertices of the graph $\mathcal{G}$ is called a \emph{dominating
set}, if every vertex in $V\setminus D$ is adjacent to at least one
vertex of $D$. The minimum cardinality of such a subset is called the
\emph{domination number} of $\mathcal{G}$ and denoted by $\gamma(\mathcal{G})$.
 For a bipartite graph
$\mathcal{G}$ with vertex bipartition $V=X\cup Y$, a subset $Z$ of $X$
is said to be a dominating set of $Y$ if any vertex in $Y$ is
adjacent to at least one vertex of $Z$, and the domination number of
$Y$ in this bipartite graph denoted by $\gamma_Y(\mathcal{G})$ is the minimum cardinality
 of a dominating set of $Y$.
Two simple nonempty graphs $\mathcal{G}=(V,E)$ and $\mathcal{G}'=(V',E')$ are
said to be \emph{isomorphic} if there exists a one to one
correspondence $\rho:V\rightarrow V'$, such that for every  $x,y\in
V$, $x$ and $y$ are adjacent in $\mathcal{G}$ if and only if $\rho(x)$
and $\rho(y)$ are adjacent in $\mathcal{G}'$. Such a $\rho$ is said to
be a \emph{graph isomorphism} from $\mathcal{G}$ to $\mathcal{G}'$. If
$\rho$ is a graph isomorphism form the graph $\mathcal{G}$ to $\mathcal{G}$,
then $\rho$ is said to be an \emph{automorphism} of the graph
$\mathcal{G}$. The set of all automorphisms of a graph $\mathcal{G}$ is denoted by
$Aut(\mathcal{G})$. For a simple graph $\mathcal{G}=(V, E)$ and any nonempty subset $A$
of $V$, we denote the neighbor of the set $A$ by $N(A)=\{x\in V\mid x
\sim a$ for some $a\in A\}$. For simplicity, if $A=\{a\}$, we
denote $N(A)$ by $N(a)$. The \emph{maximum degree} and \emph{minimum degree} of a graph
$\mathcal{G}$ are denoted by $\Delta(\mathcal{G})$ and $\delta(\mathcal{G})$, respectively,
and are defined as $\Delta(\mathcal{G})=\max\{\deg(v):v\in V\},\,
\delta(\mathcal{G})=\min\{\deg(v):v\in V\}.$\par
Throughout this paper, we assume that $\mathbb{F}_q$ is a finite field with $q$
elements, $\mathbb{V}_0=\mathbb{F}_q^n$ is the $n$-dimensional column vector space over $\mathbb{F}_q$,
 for an integer $n\geq 2$.
 The set of $n\times n$ matrices over $\mathbb{F}_q$ is
   denoted by $M_n(\mathbb{F}_q)$.  The $n\times n$ unit matrix is denoted
   by $I_n$. An element $\lambda\in\mathbb{F}_q$ is said to be an eigenvalue of a matrix
   $A\in M_n(\mathbb{F}_q)$, if there exists a non-zero vector $v\in \mathbb{V}_0$
 such that $Av=\lambda v$. Then such a non-zero vector $v$ is said to be an
 eigenvector of $A$. The set of all $n\times n$
  matrices over $\mathbb{F}_q$ with at least one eigenvector is
  denoted by $E_n(\mathbb{F}_q)$.  The subspace generated by a subset
  $U$ of any vector space $V$ is denoted by $\langle U\rangle$  and for
   the simplicity, when $U=\{u\}$ for a $u\in V$, we denote by $\langle u\rangle$ instead
   of $\langle \{u\}\rangle$ and  in this case, sometimes we say the elements of $\langle v\rangle$ are
    \emph{projectively represented by} $v$.
The \emph{general linear group of degree $n$} over the field $\mathbb{F}_q$ is
 $GL_n(\mathbb{F}_q)=\{ A\in M_n(\mathbb{F}_q): \text{ A is an invertible matrix}\}$
 with matrices product. The symmetric (permutation) group of a set $X$ is
  denoted by $S_X$ and the symmetric (permutation) group of the set $\{1,\cdots, n\}$
   is denoted by $S_n$, where $n\in \mathbb{N}$.
    For more definitions and information about the vector spaces see \cite{Hoffman}.\par
The following Lemma is used in our results along the paper
( for its similar version see \cite[page 185]{S.Roman}).
\begin{lemma} \label{changeB} Let $A\in M_n(\mathbb{F}_q)$ ba a
matrix having an eigenvector $v$ in $\mathbb{F}^n_q$.
 If we change the basis of the vector space to a basis $\mathcal{B}$, then the matrix
  $[A]_{\mathcal B}$ has the eigenvector $[v]_{\mathcal B}$ in the basis $\mathcal{B}$,
   where $[A]_{\mathcal B}$ and $[v]_{\mathcal B}$
  are the representations of $A$ and $v$ in the basis $\mathcal{B}$, respectively.
 \end{lemma}
\begin{proof} Let $A$ be a matrix in the standard basis with an eigenvector $v$
and let $P$ be the change-of-basis matrix. Then the matrix representation of $A$
in the new basis $\mathcal{B}$ is
$[A]_{\mathcal B}=P^{-1}AP$  and the vector representation of $v$ in the new basis $\mathcal{B}$ is
$[v]_{\mathcal B}=P^{-1}v$.
We show that the eigenvalue relation (i.e.  eigenvector relation)  $ Av=\lambda v$
does not depend on the basis. To see this we have:\[[A]_{\mathcal B}[v]_{\mathcal B}
=(P^{-1}AP)(P^{-1}v) =
P^{-1}Av = P^{-1}(\lambda v)
=\lambda(P^{-1}v)
=\lambda [v]_{\mathcal B}.\]
Hence$[A]_{\mathcal B}[v]_{\mathcal B}
=\lambda [v]_{\mathcal B}$.
Therefore the eigenvector relation is preserved under change of basis.
\end{proof}
\begin{remark}\label{similarmatrices} Two matrices $A , B\in M_n(\mathbb{F}_q)$ are said to be
similar if there exists an invertible matrix $P\in GL_n(\mathbb{F}_q)$ such that $PAP^{-1}=B$.
 It is well known that the similar matrices have the same set of eigenvalues (see \cite{Hoffman}).
\end{remark}
\begin{theorem}\label{intersection} Let  $n\geq 2$ be a positive integer and
 $S=\{v_1,\cdots,v_k \}$  a set of $k$ linearly independent
  non-zero vectors in $\mathbb{F}_q^n$, for an arbitrary $1\leq k\leq n$.
  Then there exists exactly $q^{n^2-kn+k}$
  matrices such that every $v_i$ is an eigenvector of all these matrices.
\end{theorem}
\begin{proof}
We want to count the number of matrices $A\in M_n(\mathbb{F}_q)$ such that every element of
$\{v_1, \cdots , v_k\}$ is an eigenvector of $A$. That means for each $A$ there exist
 $\lambda_1, \cdots ,\lambda_k\in \mathbb{F}_q$ such that
$Av_i=\lambda_i v_i$. Since $v_1,\ldots,v_k$ are linearly independent, we can extend $\{v_1,\cdots,v_k \}$
to a basis $\mathcal B=\{v_1,\ldots,v_k,  \dots, v_n\}$ of $\mathbb{F}_q^n$. Using Lemma \ref{changeB} consider the matrix
 $A$ with respect to this basis. Because $[v_i]_\mathcal{B}=e_i$ (the vector with its $i$th entry is $1$ and the others is $0$)  is the $i$th  vector of basis and
$Av_i=\lambda_i v_i$, the $i$th column of the matrix representation of $A$ in the basis
$\mathcal{B}$ is:
\[
\begin{pmatrix}
0\\
\vdots\\
\lambda_i\\
0\\
\vdots\\
0
\end{pmatrix}.
\]
Hence in the basis $\mathcal{B}$ every such matrix has the form
\[
A=
\begin{pmatrix}
\lambda_1 & 0 & \cdots &0&*&*&\cdots&* \\
0 &  \lambda_2& \cdots&0&*&*&\cdots&*\\
\vdots & \vdots & \ddots&\vdots&\vdots&\vdots&\vdots&\vdots\\
0 & 0 & \cdots &\lambda_k&*&*&\cdots&*\\
0 & 0 & \cdots &0& *&*&\cdots&*\\
\vdots & \vdots & \ddots & \vdots&\vdots&\vdots&\vdots&\vdots\\
0 & 0 & \cdots &0& *&*&\cdots&*\\
\end{pmatrix}.
\]
 and conversely, every matrix of the above form has each of the vectors $[v_i]_B$
 as one of its eigenvectors. Now, since
 $\lambda_i$ has $q$ possible choices for every $1\leq i\leq k$ and
 the remaining $n(n-k)$ entries are arbitrary elements of $\mathbb{F}_q$,
the total number of such matrices in basis $\mathcal{B}$ is $q^k\cdot q^{n(n-k)}=q^{n^2-kn+k}$.
 So using Lemma \ref{changeB} the cardinality of the
  set $\{ A\in M_n(\mathbb{F}_q): v_i \text{ is an eigenvector of \emph{A} for all }  1\leq i\leq k\}$
  is $q^{n^2-kn+k}$ and this  completes the proof.
\end{proof}
\begin{theorem}\label{count} For the field $\mathbb{F}_q$ the following statements hold:
\begin{enumerate}
  \item  For every non-zero vertex $v\in \mathbb{F}_q^n$,  the number of matrices
   $M\in M_n(\mathbb{F}_q)$ such that $v$ is an eigenvector of $M$ is $q^{n^2-n+1}$.
  \item  Let $\mathcal{B}=\{v_1,\cdots,v_n \}$ be a basis of  $\mathbb{F}^n_q$
   and $E_\mathcal{B}$  is the set of all $n\times n$ matrices with at least one
    eigenvector in $\mathcal{B}$. Then $|E_\mathcal{B}|=q^{n^2}-q^{n}\bigl(q^{n-1}-1\bigr)^n$.
\end{enumerate}
\end{theorem}
\begin{proof} (1) In Theorem \ref{intersection} let $S=\{v_1=v\}$.\\
(2) Assume $\mathcal{B}=\{v_1, v_2,\cdots, v_n\}$ is a basis for $\mathbb{F}_q^n$. For each $1\leq i\leq n$ let
$E(v_i)=\{M\in E_n(\mathbb{F}_q): v_i \text{ is an eigenvector of } M\}$.
  Clearly $E_\mathcal{B}= \bigcup_{i=1}^nE(v_i)$ and hence by inclusion- exclusion law, we have
\[|E_\mathcal{B}|=\sum_{k=1}^{n} |E(v_k)|-\sum_{1\leq i\neq j\leq n} |E(v_i)\cap E(v_j)|+ \cdots +
(-1)^{n-1} |E(v_1)\cap \cdots \cap E(v_n)|\]
Note that by Theorem \ref{intersection}, we have $|E(v_k)|=q^{n^2-n+1}$ for every $1\leq k\leq n$.
Also by Theorem \ref{intersection} for every $1\leq i\neq j\leq n$
we have $|E(v_i)\cap$$E(v_j)|= q^{n^2-2n+2}$ and
 similarly, for every subset $\{i_1,i_2,\cdots,i_k\}$ of $\{1,2,\cdots, n\}$ we have:
\[|E(v_{i_1})\cap E(v_{i_2})\cap\cdots\cap E(v_{i_k})|= q^{n^2-kn+k}.\]
So \[|E_\mathcal{B}|=\binom{n}{1}q^{n^2-n+1} -\binom{n}{2}q^{ n^2-2n+2}+ \cdots +
(-1)^{n-1}\binom{n}{n}q^{n^2-nn+n}.\]
Therefore, \[|E_\mathcal{B}|=\sum_{k=1}^{n}(-1)^{k-1}\binom{n}{k}q^{\,n^2-kn+k}.\]
Let $a_k=(-1)^{k-1}\binom{n}{k}q^{\,n^2-kn+k}$,
for every $1\le k\le n$. We want to compute
$S=$$\sum_{k=1}^{n}(-1)^{k-1}\binom{n}{k}q^{\,n^2-kn+k}$.
First we rewrite the exponent as $n^2-kn+k=n^2-k(n-1)$. Hence
$S=q^{n^2}\sum_{k=1}^{n}
(-1)^{k-1}\binom{n}{k}q^{-k(n-1)}$. Now, let
$x=q^{-(n-1)}$. Then
$
S=q^{n^2}\sum_{k=1}^{n}
(-1)^{k-1}\binom{n}{k}x^k$. Using the binomial Theorem we have:
\[(1-x)^n =\sum_{k=0}^{n}\binom{n}{k}(-1)^k x^k,\]
we obtain $\sum_{k=1}^{n}
(-1)^{k-1}\binom{n}{k}x^k
=1-(1-x)^n$. So $S=q^{n^2}\Bigl[1-(1-q^{-(n-1)})^n\Bigr].$ Then
removing negative powers, we have $S=q^{n^2}-q^{n^2-n(n-1)}\bigl(q^{n-1}-1\bigr)^n.$
Since $n^2-$$n(n-1)=n$, we get
$S=q^{n^2}-q^{n}\bigl(q^{n-1}-1\bigr)^n$.
Thus
$|E_\mathcal{B}|=$$q^{n^2}-q^{n}\bigl(q^{n-1}-1\bigr)^n$.

\end{proof}
 Along the paper the number of matrices over the
  finite field $\mathbb{F}_q$ for which three
   prescribed vectors are eigenvectors is needed, when the vectors are pairwise
   linearly independent but as a triple are linearly dependent.
\begin{lemma}\label{collinear}
Let $v_1,v_2,v_3\in \mathbb{F}_q^n$ be pairwise linearly
independent vectors such that $v_1,v_2,v_3$ are linearly dependent.
Then the number of matrices $M\in M_n(\mathbb{F}_q)$ for which every
$v_i$ is an eigenvector of $M$ is $q^{(n-1)^2}$.
\end{lemma}

\begin{proof}
Let $W=\langle v_1,v_2,v_3\rangle$. Since the vectors are pairwise
linearly independent but are linearly dependent as a triple, we have
$dim W=2$. Thus, there exist non-zero scalars $a,b\in\mathbb{F}_q$
such that $v_3=av_1+bv_2$.
Suppose that $Mv_i=\lambda_i v_i$ for $i=1,2,3$. Then
$
Mv_3=aMv_1+bMv_2
=a\lambda_1v_1+b\lambda_2v_2.$
On the other hand, since $v_3$ is also an eigenvector of $M$,
$Mv_3=\lambda_3v_3
=a\lambda_3v_1+b\lambda_3v_2$.
Since $v_1$ and $v_2$ are linearly independent and $a,b$  are non-zero, it
follows that
$\lambda_1=\lambda_2=\lambda_3=:\lambda$.
Extend $\{v_1,v_2\}$ to a basis
$\mathcal{B}=\{v_1,v_2,w_3,\ldots,w_n\}$
of $\mathbb{F}_q^n$. Relative to this basis, every matrix $M$ for
which $v_1,v_2,v_3$ are eigenvectors has the form
$M=
\begin{pmatrix}
\lambda I_2 & A\\
0 & B
\end{pmatrix},$
where $A\in M_{2\times(n-2)}(\mathbb{F}_q)$ and
$B\in M_{n-2}(\mathbb{F}_q)$. Conversely, every matrix of this form
satisfies $Mv_1=\lambda v_1,
Mv_2=\lambda v_2$, and since $v_3=av_1+bv_2$,
$Mv_3=aMv_1+bMv_2
=\lambda(av_1+bv_2)
=\lambda v_3.$
Thus, every matrix of the above form has all three vectors as
eigenvectors. There are $q$ choices for $\lambda$, $q^{2(n-2)}$ choices for $A$
and $q^{(n-2)^2}$ choices for $B$. Hence the number of such matrices is
$q\cdot q^{2(n-2)}\cdot q^{(n-2)^2}
=q^{1+2(n-2)+(n-2)^2}
=q^{(n-1)^2}.$
This completes the proof.
\end{proof}

\begin{remark}\label{rem2}
Note that if $v, w \in\mathbb{V}_0$ are non-zero vectors
 such that $w=\alpha v$ for a non-zero $\alpha \in \mathbb{F}_q$,
then for every matrix $M\in M_n(\mathbb{F}_q)$ and every $\lambda\in \mathbb{F}_q $ we have:
\[
Mv=\lambda v
\iff
Mw=\lambda w.
\]
Therefore, for $\langle v\rangle$ ( the one-dimensional subspace
 of $\mathbb{V}_0$ generated by the non-zero vector $v$), we have
  $Mv= \lambda v$ for a $\lambda\in \mathbb{F}_q$ if and only if
   $Mv\in\langle v\rangle$. This means that $\langle v\rangle$ is stabilized
   by the linear transformation $T: \mathbb{V}_0\to \mathbb{V}_0$, defined $T(x)=Mx$
   for every $x\in\mathbb{V}_0$, if and only if $v$ is an eigenvector of $M$.
\end{remark}

\begin{remark}\label{rem3}
Let $M, N \in E_n(\mathbb{F}_q)$, such that $N=\alpha M$ for a non-zero $\alpha \in \mathbb{F}_q$.
Then for every $0\neq v\in\mathbb{V}_0$, $v$ is an eigenvector of $M$ if and only if $v$ is an eigenvector
of $N$, as:
\[Mv=\lambda v
\iff
\alpha Mv=\alpha\lambda v \iff Nv=(\alpha\lambda)v
.\]
So for $\langle M\rangle$ ( the one-dimensional subspace of $M_n(\mathbb{F}_q)$
 generated by the non-zero matrix $M$), we have
   $Mv= \lambda v$ if and only if  $v$ is an eigenvector of $\alpha M$
    for every $0\neq\alpha\in \mathbb{F}_q$, if and only if every matrix in the
     $\langle M \rangle$ stabilizes the projective point $\langle v\rangle$.
     Therefore, all  elements of
    $\langle M\rangle$ have the same eigenvectors.
     Sometimes we say the elements of $\langle M\rangle$ are
    \emph{projectively represented by} $M$.
\end{remark}
Recall that the projective space $\mathbb P^{n-1}(\mathbb{F}_q)$
is defined by \[\mathbb P^{n-1}(\mathbb{F}_q)
=\{\langle v\rangle:0\neq v\in\mathbb{F}_q^n\}.\]
Thus, the points of $\mathbb P^{n-1}(\mathbb{F}_q)$ are precisely the
one-dimensional subspaces of $\mathbb{F}_q^n$.\\
Now, by Remarks \ref{rem2} and  \ref{rem3} we have the following Lemma:
\begin{lemma}\label{projectivepoints}
For the field $\mathbb{F}_q$ the following statements hold:
\begin{enumerate}
\item Each point of $\mathbb P^{n-1}(\mathbb{F}_q)$ corresponds to a
 one-dimensional subspace of $\mathbb{F}_q^n$.
\item Vectors differing by a non-zero scalar represent the same projective point.
\item $u\in \langle v\rangle$ is an eigenvector of the matrix $M$  if and only if every element of
$\langle M\rangle$ as a linear transformation of $\mathbb{V}_0$ stabilize $\langle v\rangle$ .
\end{enumerate}
\end{lemma}
The observations in Remarks \ref{rem2}, \ref{rem3} and Lemma \ref{projectivepoints}
 together with the role of eigenvalue relations in projective
 geometry and the invariance of algebraic structures under automorphisms,
 motivated us to define of a new graph as following.
\begin{definition}\label{a1} Let $\mathbb{F}_q$ be a finite field with $q$ elements and let
$\mathbb{V}_0= \mathbb{F}_q^n$ be the $n$-dimensional column vector space over
$\mathbb{F}_q$, for an integer $n\geq 2$. Let \[E_n(\mathbb{F}_q)=\{ M\in M_n(\mathbb{F}_q):
 M\text{ is a non-zero matrix with an eigenvector in $\mathbb{V}_0$} \}\]
 and
  $\mathcal{E}^{n-1}(\mathbb{F}_q)=\{\langle M\rangle: M \in$$E_n(\mathbb{F}_q)\}$
  the set of all one-dimensional subspaces of $M_n(\mathbb{F}_q)$ generated by the
  elements of $E_n(\mathbb{F}_q)$. The \emph{Spectral Incidence Graph} of
 $\mathbb{V}_0$,
 denoted by $ \mathbf{SIG}(\mathbb{V}_0)$, is a bipartite graph whose vertex set $V$
 is partitioned into two parts
 $\mathcal{E}^{n-1}(\mathbb{F}_q)$ and $\mathbb{P}^{n-1}(\mathbb{F}_q) $,
where two vertices $\langle M\rangle \in \mathcal{E}^{n-1}(\mathbb{F}_q)$
and $\langle v \rangle\in \mathbb{P}^{n-1}(\mathbb{F}_q)$ are
adjacent if and only if $v$ is an eigenvector of $M$ (i.e. there exists a
 $\lambda \in \mathbb{F}_q$  such that $Mv=\lambda v$). So
\[\langle M\rangle \sim \langle v\rangle \text{ is an edge of } \mathbf{SIG}(\mathbb{V}_0)
\quad\Longleftrightarrow\quad
v \text{ is an eigenvector of } M.\]
\end{definition}
For an arbitrary element $\lambda\in \mathbb{F}_q $, the \emph{elementary Jordan matrix} of order $k$ with eigenvalue $\lambda$ is the following $k\times k$ matrix : \[
J_k(\lambda)=
\begin{pmatrix}
\lambda & 1 & 0 &\cdots & 0\\
0 & \lambda &1& \cdots & 0\\
\vdots & \vdots & \ddots & \vdots&\vdots\\
0 & \vdots & \cdots& \lambda & 1\\
0 & \cdots&  \cdots &0& \lambda
\end{pmatrix}
\] (see \cite[Corollary 4.4]{Hungerford}), when $k=1$, it is called a \emph{trivial}
 elementary Jordan matrix, whereas when
  $k>1$
 it is called a \emph{nontrivial elementary Jordan matrix}.
In some references such as \cite{Hoffman} the entries $1$ are under the main diagonal. This
difference does not effect our results since two types are similar matrices.

For a matrix $M\in M_n(\mathbb{F}_q)$, let
  $\emph{Spec}(M)= \{\lambda \in \mathbb{F}_q: det(M-\lambda I_n)=0 \}$ and
   $d_{\lambda}$ the dimension of
    $E_{\lambda}=\{v\in \mathbb{F}_q^n: (M-\lambda I_n)v=0\}$ (i.e. $d_{\lambda}=dim(ker(M-\lambda I_n)).)$
Clearly we have $|\emph{Spec}(M)|\leq q.$
\begin{theorem}\label{regular} For the graph $ \mathbf{SIG}(\mathbb{V}_0)$ the following statements hold:
\begin{enumerate}
  \item $deg(\langle M\rangle)= \sum_{\lambda \in \emph{\emph{Spec}}(M)} \frac{q^{d_{\lambda}}-1}{q-1}$ and
    $1\leq deg(\langle M\rangle)\leq \frac{q^n-1}{q-1}$, for every matrix $M\in E_n(\mathbb{F}_q)$.
  \item For every  $\langle v\rangle\in \mathbb P^{n-1}(\mathbb{F}_q)$,
  $deg(\langle v\rangle)=\frac{q^{n^2-n+1}-1}{q-1}$.
  \item $\triangle (\mathbf{SIG}(\mathbb{V}_0))= \frac{q^{n^2-n+1}-1}{q-1} $
  and $\delta (\mathbf{SIG}(\mathbb{V}_0))=1$.
\end{enumerate}
\end{theorem}
\begin{proof}
\indent(1) For every matrix $M \in E_n(\mathbb{F}_q)$, let $E_{\lambda}= ker( M-\lambda I_n)$
 and $d_{\lambda}=dim E_{\lambda}$. Then since the vertices are one-dimensional,
   each $E_{\lambda}$  contributes exactly
  $\frac{q^{d_{\lambda}}-1}{q-1}$ vertices contained to the projective  part
   $\mathbb{P}^{n-1}(\mathbb{F}_q)$.  Since eigenspaces corresponding to distinct eigenvalues
   intersect only in the zero vector, the contributions are disjoint. Therefore,
 $deg(\langle M\rangle)= \sum_{\lambda \in \emph{Spec}(M)} \frac{q^{d_{\lambda}}-1}{q-1}$.
  Every matrix $A\in E_n(\mathbb{F}_q)$ has at least one non-zero
   eigenvector over $\mathbb{F}_q$. The identity matrix $I_n$ has
    every non-zero vector as an eigenvector with eigenvalue 1.
     So every vertex $\langle v\rangle$ in part $\mathbb {P}^{n-1}(\mathbb{F}_q)$
      is adjacent to $I_n$. Hence $deg(\langle I_n\rangle)=\frac{q^n-1}{q-1}$ is
      the maximum degree in part $\mathcal{E}^{n-1}(\mathbb{F}_q)$. So  for every
       $M\in E_n(\mathbb{F}_q)$, we have $1\leq deg(\langle M\rangle)\leq |\mathbb{P}^{n-1}(\mathbb{F}_q)|\leq \frac{q^n-1}{q-1}$.\\
\indent (2) By Theorem \ref{count} part (1) the number
 of matrices $0\neq A\in E_n(\mathbb{F}_q)$ such that
$v$ is an eigenvector of $A$ is $q^{n^2-n+1}-1$.
Thus the cardinality of the set
 $\{\langle A\rangle \in \mathcal{E}^{n-1}(\mathbb{F}_q):
  v \text{ is an eigenvector of } A  \text{ and } A\neq 0 \}$ is
$\frac{q^{n^2-n+1}-1}{q-1}$ and hence $deg(\langle v\rangle)=\frac{q^{n^2-n+1}-1}{q-1}$.\\
(3)  Finally   $\delta(\mathbf{SIG}(\mathbb{V}_0))=1$ as for every $\lambda \in \mathbb{F}_q$
the elementary Jordan matrix $J_n(\lambda)$
has eigenspace dimension 1. Therefore, $deg(\langle J_n(\lambda\rangle))=1=\delta (\mathbf{SIG}(\mathbb{V}_0)$.
 Since $n\geq 2$, we have
$\frac{q^n-1}{q-1} <\frac{q^{n^2-n+1}-1}{q-1}$ and using parts (1)
 and (2) we obtain $\triangle (\mathbf{SIG}(\mathbb{V}_0))=
  \frac{q^{n^2-n+1}-1}{q-1}.$
\end{proof}
\begin{lemma}\label{degree1}
For every non-zero vector $v\in \mathbb{F}_q^n$, there exists at least
  one matrix $A\in$$E_n(\mathbb{F}_q)$ such that the only eigenspace of
   $A$ is  $\langle v\rangle$. Equivalently,
for every $L\in$$\mathbb{P}^{n-1}(\mathbb{F}_q)$, there exists a matrix
 $A\in E_n(\mathbb{F}_q)$ whose only eigenspace is $L$.
\end{lemma}
\begin{proof}
  Suppose $L=\langle v\rangle$ is a one-dimensional subspace of $\mathbb{F}_q^n$. Since $v\neq0$
we can choose a basis $\mathcal{B}=\{v=v_1,v_2,\dots,v_n\}$ for $\mathbb{F}_q^n$.
 We wish to construct a matrix $M$ whose only eigenspace is $L$.
 With respect to this basis, the elementary Jordan matrix $J_n(\lambda)$
  has the unique one-dimensional eigenspace $ L=\langle [v]_\mathcal{B}\rangle$.
 Let $P$ be the change-of-basis matrix whose columns are
$v_1,\ldots,v_n$. Define $A=PJ_n(\lambda)P^{-1}$. So using
Lemma \ref{changeB} and Remark \ref{similarmatrices}, the unique eigenspace of $A$ is also $\langle v\rangle$ in the standard basis and this  completes the proof.
\end{proof}
\begin{theorem}\label{diam}
  The graph $\mathbf{SIG}(\mathbb{V}_0)$ is a connected graph with
  $\frac{q^n-1}{(q-1)^2}(q^{n^2-n+1}-1)$ edges and
$diam(\mathbf{SIG}(\mathbb V_0))=4.$
\end{theorem}
\begin{proof}
  For every $0\neq v\in \mathbb{F}_q^n $  we have  $ I_n v = v$. So $\langle I_n \rangle$
   is adjacent to every vertex $\langle v\rangle$. Hence every two distinct vertices
   $\langle v_1\rangle$ and $\langle v_2\rangle$ of part
    $\mathbb{P}^{n-1}(\mathbb{F}_q)$ are connected by a path of length 2.
    Let $\langle M \rangle, \langle N \rangle \in \mathcal{E}^{n-1}(\mathbb{F}_q)$ such
     that $\langle M\rangle\neq \langle N\rangle$.
     If $M$ and $N$ have a common eigenvector $v\in \mathbb{V}$, then
      $\langle M\rangle\sim \langle v \rangle\sim \langle N\rangle$ is a path of length 2.
       If $M$ and $N$  have no common eigenvector
   then we have the path
   $\langle M\rangle\sim \langle v_M \rangle\sim \langle I_n\rangle\sim
   \langle v_N\rangle \sim \langle N\rangle$ is a path of length 4, where $v_M$ and $v_N$ are eigenvectors
   of $M $ and $N$, respectively (note that $\langle v_M\rangle\neq\langle v_n\rangle$). If for
     $\langle M\rangle\in \mathcal{E}^{n-1}(\mathbb{F}_q) $ and
    $\langle v\rangle\in \mathbb{P}^{n-1}(\mathbb{F}_q)$, $\langle M\rangle$
     and $\langle v\rangle$ are not adjacent, then we have the path
      $\langle M\rangle\sim \langle v_M \rangle\sim \langle I_n\rangle\sim \langle v\rangle $,
       where $v_M$ is an eigenvector of $M$. So the graph is connected and $diam(\mathbf{SIG}(\mathbb V_0))\leq 4.$
           Now, we want to determine the diameter of $\mathbf{SIG}(\mathbb{V}_0)$ and then the number of its edges.
 Note that  $q\geq2$ and $n\geq2$, hence $|\mathbb{P}^{n-1}(\mathbb{F}_q)|=\frac{q^n-1}{q-1}\geq 3$. By Lemma \ref{degree1}
for every vertex $\langle v\rangle\in \mathbb{P}^{n-1}(\mathbb{F}_q)$  there exists a matrix vertex (i.e. one-dimensional subspace of $M_n(\mathbb{F}_q)$)
$\langle M_v\rangle\in$$\mathcal{E}^{n-1}(\mathbb{F}_q)$ such that  $\langle M_v\rangle$ only is adjacent to $\langle v\rangle$(i.e. $N(\langle M_v\rangle)= \{\langle v\rangle\}$). So choose two distinct vertices
  $\langle v\rangle\neq\langle u\rangle$ in $\mathbb{P}^{n-1}(\mathbb{F}_q)$.   Then there
  exist two matrix vertices $\langle M_v\rangle$ and $\langle M_u\rangle$ in $\mathcal{E}^{n-1}(\mathbb{F}_q)$  such that $N(\langle M_u\rangle)= \{\langle u\rangle\}$
   and $N(\langle M_v\rangle)= \{\langle v\rangle\}$. Clearly $\langle M_u\rangle\neq\langle M_v\rangle$.
   Then the path $\langle M_u\rangle\sim \langle u\rangle\sim \langle I_n\rangle\sim\langle v\rangle
   \sim\langle M_v\rangle$ is a path with the minimum length between two vertices $\langle M_u\rangle$ and $\langle M_v\rangle.$  Therefore $diam(\mathbf{SIG}(\mathbb V_0))=4$. Now, since  $|\mathbb{P}^{n-1}(\mathbb{F}_q)|=\frac{q^n-1}{q-1}$ and by Theorem \ref{regular} every
     vertex $\langle v\rangle\in \mathbb{P}^{n-1}(\mathbb{F}_q)$  has the degree $\frac{q^{n^2-n+1}-1}{q-1}$, the graph $\mathbf{SIG}(\mathbb{V}_0)$ has  $\frac{q^n-1}{(q-1)^2}(q^{n^2-n+1}-1)$ edges and this  completes the proof.
\end{proof}

\begin{proposition}\label{scalarmatrix} For a matrix $M\in E_n(\mathbb{F}_q)$, $\langle M\rangle$ is
 adjacent to every vertex $\langle v \rangle \in\mathbb{P}^{n-1}(\mathbb{F}_q)$  if and
only if $M$ is a scalar matrix (i.e. $M= \lambda I_n$  for some $\lambda \in \mathbb{F}_q$).
 In particular, in part $\mathcal{E}^{n-1}(\mathbb{F}_q)$ only $\langle I_n\rangle$  has the maximum degree $\frac{q^n-1}{q-1}.$
\end{proposition}
\begin{proof}
Assume that every non-zero vector is an eigenvector of $M$ and $\{e_1,e_2, \ldots ,e_n\}$ is the standard basis for $\mathbb{F}_q^n$. Then for each $e_i$ there exists a $\lambda_i\in \mathbb{F}_q$ such that $Me_i=\lambda_i e_i$.  So $M=\begin{pmatrix}
\lambda_1 & 0 & \cdots &0 \\
0 &  \lambda_2& \cdots&0\\
\vdots & \vdots & \ddots&\vdots&\\
0 & 0 & \cdots &\lambda_n
\end{pmatrix}\text{is a diagonal matrix}$. Also for some $\lambda \in \mathbb{F}_q$ we have
$M \begin{pmatrix}
     1 \\
     1 \\
     \vdots \\
     1
   \end{pmatrix} =\lambda \begin{pmatrix}
     1 \\
     1 \\
     \vdots \\
     1
   \end{pmatrix}$. So $\begin{pmatrix}
     \lambda_1 \\
     \lambda_2 \\
     \vdots \\
     \lambda_n
   \end{pmatrix} = \begin{pmatrix}
     \lambda \\
     \lambda \\
     \vdots \\
     \lambda
   \end{pmatrix} $. Therefore, $\lambda_1=\lambda_2=\ldots=\lambda_n= \lambda$ and hence $M=\lambda I_n$
   is a scalar matrix. Note that every scalar matrix $\lambda I_n$ has all non-zero vectors  as its eigenvectors. Clearly we have $q-1$ such matrices.
 So only the vertex of largest degree in part $\mathcal{E}^{n-1}(\mathbb{F}_q)$ is
  precisely $\langle\lambda I_n\rangle =\langle I_n\rangle$ with $deg(\langle I_n\rangle)= |\mathbb{P}^{n-1}(\mathbb{F}_q)|=\frac{q^n-1}{q-1}$.
\end{proof}
 Consider a bipartite graph $\mathcal{G}$ with the vertex set $V= A\cup B$,  where $A$ and $B$
are the related partitions of $V$. Recall that a subset $D\subseteq B$ dominates $A$ if
every vertex of $A$ is adjacent to at least one vertex of $D$.
 The{ domination number of $A$ in the graph $\mathcal{G}$} denoted by $\gamma_A(\mathcal{G})$ which
 is defined as: \[\gamma_A(\mathcal{G})=min\{|D|:D\subseteq B \text{ such that \emph{D} dominates \emph{A}}\}.\]
For every bipartite graph $\mathcal{G}$ with bipartition vertex set $V= A\cup B$,
in general $\gamma(\mathcal{G})\leq \gamma_A(\mathcal{G})+\gamma_B(\mathcal{G})$. In the following result we
determined $\gamma(\mathbf{SIG}(\mathbb{V}_0))$ and we see that in the mentioned inequality,
the proper case holds for the bipartite graph $\mathbf{SIG}(\mathbb{V}_0)$.

\begin{theorem}\label{a2} For the graph $\mathbf{SIG}(\mathbb{V}_0)$ the following statements hold:
\begin{enumerate}
\item $\gamma_{\mathcal{E}^{n-1}(\mathbb{F}_q)}(\mathbf{SIG}(\mathbb{V}_0))=
|\mathbb{P}^{n-1}(\mathbb{F}_q)|= \frac{q^n-1}{q-1}$.
\item$\gamma_{\mathbb{P}^{n-1}(\mathbb{F}_q)}(\mathbf{SIG}(\mathbb{V}_0))=1$.
\item$\gamma(\mathbf{SIG}(\mathbb{V}_0))=\frac{q^n-1}{q-1}$.
\end{enumerate}
  \end{theorem}
\begin{proof}
\indent(1) Clearly the set $\mathbb{P}^{n-1}(\mathbb{F}_q)$ dominates
the set $\mathcal{E}^{n-1}(\mathbb{F}_q)$. Note that  by Lemma \ref{degree1}
for every  vertex $\langle v\rangle \in \mathbb{P}^{n-1}(\mathbb{F}_q) $
there exists a matrix $M\in E_n(\mathbb{F}_q)$, such that the only
 eigenspace of $M$ is $\langle v\rangle$. So the only neighbor of
  $\langle M\rangle$ in the graph $\mathbf{SIG}(\mathbb{V}_0)$ is $\langle v\rangle$.
  Therefore, the set $\mathbb{P}^{n-1}(\mathbb{F}_q)\setminus\langle v\rangle$ doesn't dominate
   the set $\mathcal{E}^{n-1}(\mathbb{F}_q)$. So $\mathbb{P}^{n-1}(\mathbb{F}_q)$
    is a set with  minimum cardinality which dominates $\mathcal{E}^{n-1}(\mathbb{F}_q)$ and hence
$\gamma_{\mathcal{E}^{n-1}(\mathbb{F}_q)}(\mathbf{SIG}(\mathbb{V}_0))
=|\mathbb{P}^{n-1}(\mathbb{F}_q)|= \frac{q^n-1}{q-1}$.\\
\indent (2) Since every non-zero vector $v\in \mathbb{F}_q^n$ is an
 eigenvector of matrix $I_n$, so $\{\langle I_n\rangle\}$ is a dominating
 set with the minimum cardinality for $\mathbb{P}^{n-1}(\mathbb{F}_q)$. Therefore,
$\gamma_{\mathbb{P}^{n-1}(\mathbb{F}_q)}(\mathbf{SIG}(\mathbb{V}_0))= |\{\langle I_n\rangle\}|=1$.\\

\indent (3)
 It is clear that the part $\mathbb{P}^{n-1}(\mathbb{F}_q)$ is a dominating
  set for the graph $\mathbf{SIG}(\mathbb{V}_0)$. Hence
$\gamma(\mathbf{SIG}(\mathbb{V}_0))\le |\mathbb{P}^{n-1}(\mathbb{F}_q)|=\frac{q^n-1}{q-1}$.
 By Lemma \ref{degree1}, for each  $\langle v\rangle\in$$\mathbb{P}^{n-1}(\mathbb{F}_q) $,
  there exists at least one vertex
in $\mathcal{E}^{n-1}(\mathbb{F}_q) $
 such that whose unique eigenspace is $\langle v\rangle$.
Thus, for each vertex $\langle v\rangle $, we can choose a vertex
$\langle A_v\rangle $ whose unique
eigenspace is $\langle v\rangle$.
Then $N(\langle A_v\rangle)=$$\{\langle v\rangle\}$. Note that
every dominating set of graph $\mathbf{SIG}(\mathbb{V}_0)$ must contain at least one vertex from each pair
$\{\langle v\rangle,\langle A_v\rangle\}$ for every $\langle v\rangle\in \mathbb{P}^{n-1}(\mathbb{F}_q)$.
Therefore, every dominating set has the cardinality at least $|\mathbb{P}^{n-1}(\mathbb{F}_q)|$ and hence
$\gamma(\mathbf{SIG}(\mathbb{V}_0))\ge |\mathbb{P}^{n-1}(\mathbb{F}_q)|$.
 Combining  two inequalities yields $\gamma(\mathbf{SIG}(\mathbb{V}_0))=|\mathbb{P}^{n-1}(\mathbb{F}_q)|=\frac{q^n-1}{q-1}$.
\end{proof}
\section{\bf On the Twin Classes of $\mathbf{SIG}(\mathbb{V}_0$)}\label{section2}
 Let $\mathcal{G}$ be a simple graph  with the vertex set $V$.
  Two vertices $x$ and $y$ of $\mathcal{G}$ are
called \emph{twins} if $x$ and $y$ have the same neighbors
(i.e. $N(x) = N(y)$). Consider a binary relation $\mathcal{R} $ on the vertex set
$V$ as:\[ x\mathcal{R}y \iff N(x)=N(y),\text{ (i.e.
$x\mathcal{R}y$ if and only if $x$ and $y$ are twins)}.\]
 It is known that $\mathcal{R} $ is an equivalence relation on $V$. For a vertex $v\in V$, the set
$t(v)= \{x \in V: N(v)=N(x)\}$ is the equivalence class of twins (i.e. twin class) containing
  $v$. Clearly for each pair of vertices $u,v \in V$,  $t(v)=t(u)$  if and only if $N(v)=N(u)$.
  If $t(v)\neq \{v\}$ then $t(v)$ is said to be a \emph{nontrivial  twin class}.
  If $t(v)=\{v\}$ then $t(v)$ is said to be a \emph{trivial twin class}.
  Let $\tau: V\to V$ be a function on the vertex set $V$ of a graph $\mathcal{G}$, which
stabilizes every  twin class of $V$ and acts as a permutation of each
twin class.
It is easy to see that $\tau(x) \sim
\tau(y)$ is an edge  if and only if $x \sim y$ is an edge. Thus $\tau$ is a graph automorphism
of $\mathcal{G}$.
The aim of this section is investigating the twin classes of $\mathbf{SIG}(\mathbb{V}_0)$.
In Theorem \ref{twinpoints} we characterize the nontrivial twin classes of $\mathbf{SIG}(\mathbb{V}_0)$
 and we show that in the part $\mathbb{P}^{n-1}(\mathbb{F}_q)$ we have no nontrivial twin class.\par
For every  nonempty subset $U$ of $\mathbb{F}^n_q$
 we define the projectivization of $U$ by $\mathbb{P}(U)=\{ \langle u\rangle : 0\neq u \in U\}$.
For a matrix $A\in E_n(\mathbb{F}_q)$, define
$\mathcal E(A)=
\bigcup_{\lambda\in \emph{Spec}(A)}ker(A-\lambda I)$. Thus, $\mathcal E(A)$
consists of all eigenvectors of $A$ together with the zero vector. The following lemma is immediate.
 \begin{lemma} \label{twinmatvert}  Two matrix vertices
 $\langle A\rangle$ and $\langle B\rangle$ of  $\mathbf{SIG}(\mathbb{V}_0)$
 are twins if and only if $\mathcal E(A)=\mathcal E(B)$.
 \end{lemma}
 \begin{theorem}\label{thm:similar-twin-class}
Assume $A,B\in E_n(\mathbb{F}_q)$ are two similar matrices. Then
$|t(\langle A\rangle)|
=|t(\langle B\rangle)|$ and $|N(\langle A\rangle)|=|N(\langle B\rangle)|.$
\end{theorem}

\begin{proof}
Since $A$ and $B$ are similar, there exists an invertible matrix
$P\in GL_n(\mathbb{F}_q)$ such that
$B=P^{-1}AP.$
Define
$\Phi:t(\langle A\rangle)\longrightarrow t(\langle B\rangle)
\quad\text{by}\quad
\Phi(\langle X\rangle)
=
\langle P^{-1}XP\rangle.$
 As
 in the proof of Lemma \ref{changeB},  conjugation by
 $P$ sends every eigendirection of $X$ onto the
corresponding eigendirection of $P^{-1}XP$. It is easy to show that
$\Phi$ is a bijection and hence
$|t(\langle A\rangle)|=|t(\langle B\rangle)|.$
For the second equality  the map \[\gamma :N(\langle A\rangle)\rightarrow N(\langle B\rangle)
\text{ by } \gamma(\langle v\rangle)= \langle P^{-1} v\rangle\] is a bijection and this completes the proof.
\end{proof}
\begin{lemma}\label{DistinctEigenvaluesLemma}
Let $M\in E_n(\mathbb{F}_q)$ and suppose
$
N(\langle M\rangle)
=
\{\langle v_1\rangle,\ldots,\langle v_k\rangle\},
$
where $\{v_1,\ldots,v_k\}$ is a linearly independent set of vectors in
$\mathbb{F}_q^n$. Then the vectors $v_1,\ldots,v_k$ correspond to pairwise distinct
eigenvalues of $M$.
\end{lemma}
\begin{proof}
Suppose, to the contrary, that there exist distinct indices $i,j\in
\{1,\ldots,k\}$ such that
$Mv_i=\lambda v_i
\quad\text{and}\quad
Mv_j=\lambda v_j
$ for some $\lambda\in \mathbb{F}_q$.
Since $v_i$ and $v_j$ are linearly independent,
$v_i+v_j\neq0$. Moreover,
\[
M(v_i+v_j)
=
Mv_i+Mv_j
=
\lambda v_i+\lambda v_j
=
\lambda(v_i+v_j).
\]
Hence $v_i+v_j$ is also an eigenvector of $M$ corresponding to the
eigenvalue $\lambda$. Therefore, $\langle v_i+v_j\rangle\in N(\langle M\rangle),$ a contradiction,
as $
\langle v_i+v_j\rangle\notin  \{\langle v_1\rangle,\ldots,\langle v_k\rangle\}$
Therefore,  the corresponding eigenvalues are pairwise distinct.
\end{proof}

Recall that for a non-zero subspace $W$ of $\mathbb{F}_q^n$, the set
$\mathbb{P}(W)=\{\langle v\rangle: \, 0\neq v\in W\}$ is called the projectivization of $W$.
\begin{theorem}\label{accursneighbor}
Let $W$ be a non-zero subspace of $\mathbb{F}_q^n$ with
$dim W=k$ .
Then there exists a matrix $M\in E_n(\mathbb{F}_q)$ whose only eigenvalue is 1,
such that
$N(\langle M\rangle)=\mathbb{P}(W).$
In other words, every projective subspace of dimension $k-1$ of
$\mathbb{P}^{n-1}(\mathbb{F}_q)$ occurs as the neighborhood of a matrix vertex of
$\mathbf{SIG}(\mathbb{V}_0)$.
\end{theorem}

\begin{proof}
Let $dim W=k.$
Choose a basis
$\{w_1,\ldots,w_k\}$
of $W$ and extend it to a basis of $\mathbb{F}_q^n$ as
$\{w_1,\ldots,w_k,u_1,\ldots,u_{n-k}\}.$
Define a linear transformation $M:\mathbb{V}_0\rightarrow \mathbb{V}_0$ by
$M(w_i)=w_i,$ for $1\leq i\leq k-1,$
and define its restriction to the subspace
$U=\langle w_k,u_1,\ldots,u_{n-k}\rangle$
by $M|_U=J_{n-k+1}(1).$
Hence, if $k< n$ with respect to the above basis of $\mathbb{F}^n_q$, we have
$M$ is similar to $ I_{k-1}\oplus J_{n-k+1}(1)$ and if $k=n$ then $M$ is similar to $I_n$.
Now, $M-I_n$ acts as zero transformation on
$\langle w_1,\ldots,w_{k-1}\rangle$ and $ker(J_{n-k+1}(1)-I)=\langle w_k\rangle$.
So $dim (ker(J_{n-k+1}(1)-I))=1.$ Moreover, $1$ is the only eigenvalue of $M$. Hence every eigenvector
of $M$ belongs to the eigenspace corresponding to the eigenvalue $1$.
Therefore, $ dim (ker(M-I_n))=(k-1)+1=k.$
Thus, $\mathcal{E}(M)=ker(M-I_n)=W.$
So,
$N(\langle M\rangle)
=\{\langle v\rangle:v\in \mathcal{E}(M)\setminus\{0\}\}.$
Consequently,
$N(\langle M\rangle)
=\mathbb{P}(W).$
\end{proof}
\begin{theorem}\label{thm:scalarcriterion}Let $B=\{v_1,v_2,\ldots,v_n\}$ be a basis of $\mathbb F_q^n$, and
 let $u=\sum_{i=1}^n a_i v_i$ be a non-zero vector such that $a_i\neq 0$ for every $1\le i\le n$.
 Suppose that $\{\langle v_1\rangle,\ldots,\langle v_n\rangle,\langle u\rangle\} \subseteq N(\langle M\rangle)$
for a matrix vertex $\langle M\rangle$ of $\operatorname{SIG}(V_0)$.Then $M$ is a scalar matrix. Consequently,
$N(\langle M\rangle)=\mathbb P^{n-1}(\mathbb F_q)$ and
$|N(\langle M\rangle)| =\frac{q^n-1}{q-1}.$
\end{theorem}

\begin{proof}Since $\langle v_i\rangle\in N(\langle M\rangle)$ for every $1\le i\le n$, there exist
 $\lambda_1,\ldots,\lambda_n\in\mathbb F_q$ such that
$Mv_i=\lambda_i v_i$ for every $1\le i\le n$.
Since $\langle u\rangle\in N(\langle M\rangle)$, there exists a $\lambda\in\mathbb F_q$ such that
$Mu=\lambda u$. On the other hand,
$Mu =M\left(\sum_{i=1}^n a_i v_i\right) =\sum_{i=1}^n a_i\lambda_i v_i$, whereas
$\lambda u=\sum_{i=1}^n a_i\lambda v_i$. Since $v_1,\ldots,v_n$ is a basis,
 comparison of coefficients gives
$a_i(\lambda_i-\lambda)=0,$ for every $1\le i\le n$. Because $a_i\neq0$ for every $i$, we obtain
$\lambda_1=\cdots=\lambda_n=\lambda$. Thus
$Mv_i=\lambda v_i$, for every $1\le i\le n$, and hence
$M=\lambda I_n$. Therefore, every non-zero vector of $\mathbb F_q^n$ is an eigenvector of $M$, and consequently
$N(\langle M\rangle)=\mathbb P^{n-1}(\mathbb F_q)$. The number
of one-dimensional subspaces of $\mathbb F_q^n$ is
$\frac{q^n-1}{q-1}$, which proves the result.
\end{proof}
In Section \ref{section3} for determining the group $Aut(\mathbf{SIG}(\mathbb{V}_0))$ for $n=2$,
  the following Lemma has a substantial role.
  \begin{lemma}\label{caseN(M)} Let $n=2$.  Then the following statements hold:
\begin{enumerate}
  \item For each pair of distinct vertices $\langle u\rangle$, $\langle v\rangle$
   of the part $\mathbb{P}^1(\mathbb{F}_q)$ there exist exactly q matrix vertices
   $\langle M\rangle \in \mathcal{E}^{1}(\mathbb{F}_q)$ such that
   $N(\langle M\rangle)= \{ \langle u\rangle, \langle v\rangle\}$ and every such matrix
   $M$  is similar to $ J_1(\lambda)\oplus J_1(\mu),$  for two scalars  $\lambda\neq\mu \in \mathbb{F}_q$.
  \item  For every $\langle v\rangle \in\mathbb{P}^1(\mathbb{F}_q)$ there exist exactly $q$ matrix vertices
   $\langle M\rangle \in \mathcal{E}^{1}(\mathbb{F}_q)$ such that
   $N(\langle M\rangle)= \{\langle v\rangle\}$ and every such  matrix $M$ is similar
   to $J_2(\lambda)$ for a scalar $\lambda \in \mathbb{F}_q$.
  \item If $|N(\langle M \rangle)|\geq 3$ then $M=\lambda I_2$ for some $0\neq\lambda\in \mathbb{F}_q$,
   so $N(\langle M\rangle)=\mathbb{P}^1(\mathbb{F}_q)$ and hence the
   number of such a vertex $\langle M\rangle$ with
      $|N(\langle M\rangle)|\geq 3$ is 1
       and that is $\langle M\rangle=\langle I_2\rangle$.
  \item For every vertex $\langle M\rangle$ in part $\mathcal{E}^1(\mathbb{F}_q)$
   exactly one of the above three
  cases holds and hence we have:
  \[|t(\langle M\rangle)|=
  \begin{cases}
    q, & \mbox{if } |N(\langle M\rangle)|=2, \\
    q, & \mbox{if } |N(\langle M\rangle)|=1, \\
    1, & \mbox{if } |N(\langle M\rangle)|=q+1,
  \end{cases}\]
   \end{enumerate}
\end{lemma}
\begin{proof}
\begin{enumerate}
  \item In the basis $\mathcal{B}=\{u,v\}$ of $\mathbb{F}_q^2$, consider an arbitrary
  matrix $M$, where $[M]_\mathcal{B}=J_1(\lambda)\oplus J_1(\mu),$
   for scalars $\lambda\neq\mu \in \mathbb{F}_q$.
     Clearly $[M]_\mathcal{B}$ has $\langle[u]_B\rangle$ and $\langle[v]_B\rangle$
     as its all eigendirections.
    So by Lemma \ref{changeB} in the standard basis
    $N(\langle M\rangle)= \{ \langle u\rangle, \langle v\rangle\}$.
  Also, using Lemma \ref{DistinctEigenvaluesLemma} every matrix
   $M$ satisfying $N(\langle M\rangle)=\{\langle u\rangle,\langle v\rangle\}$
    is diagonalizable with distinct eigenvalues $\lambda,\mu\in\mathbb F_q$, and,
     relative to the basis $\mathcal B=\{u,v\}$, has the form
     $J_1(\lambda)\oplus J_1(\mu)$, where $\lambda\ne\mu$. There are $q$ choices for $\lambda$ and $q-1$
     choices for $\mu$, so we have
 $q(q-1)$ such matrices in this fixed basis.
       Since each projective matrix vertex contains exactly
    $q-1$ non-zero scalar multiples, dividing by $q-1$,
    we have $|t(\langle M\rangle)|=\frac{q(q-1)}{q-1}=q.$\\
 \item Suppose that
$N(\langle M\rangle)=\{\langle v\rangle\}$
for some $\langle v\rangle\in\mathbb P^1(\mathbb F_q)$.
 Choose $u\in\mathbb V_0$ such that
$\mathcal B=\{v,u\}$
is a basis of $\mathbb V_0$.
 Since $\langle v\rangle$ is the only eigendirection of $M$,
 the matrix $M$ has exactly one eigenvalue, say $\lambda\in\mathbb F_q$,
 and its eigenspace is
$E_\lambda(M)=\langle v\rangle.$
Consequently, $M$ is not diagonalizable. Hence, with respect to the basis
 $\mathcal B$, its Jordan canonical form is
 $J_2(\lambda)=
\begin{pmatrix}
\lambda & 1\\
0 & \lambda
\end{pmatrix}.$
More explicitly, every matrix having $\langle v\rangle$ as its unique eigendirection has
 the form
$
\begin{pmatrix}
\lambda & a\\
0 & \lambda
\end{pmatrix}$  with respect to the basis $\mathcal B$,  where
$\lambda\in\mathbb F_q$ and $0\neq a\in\mathbb F_q.$
Indeed, the condition $a\ne0$ is precisely what guarantees that the eigenspace is one-dimensional.
 Thus there are $q(q-1)$ such matrices. Since two non-zero scalar multiples
determine the same matrix vertex, each matrix vertex
 contains exactly $q-1$ non-zero scalar multiples. Therefore,
  the number of matrix vertices having ${\langle v\rangle}$ as their neighborhood is
$|t(\langle M\rangle)|=\frac{q(q-1)}{q-1}=q.$

 \item Let $\{\langle v_1 \rangle, \langle v_2\rangle, \langle u\rangle\}\subseteq $$N(\langle M\rangle)$,
  where $ \langle v_1 \rangle, \langle v_2\rangle $ and $\langle u\rangle $ are three distinct vertices.
 Then clearly $\{ v_1, v_2\}$ is a basis for $\mathbb{F}_q^2$ and
 $u=a_1v_1+a_2v_2$ for two non-zero scalars $a_1, a_2 \in \mathbb{F}_q$. So
  by Theorem \ref{thm:scalarcriterion} $M$ is a scalar
 matrix. Hence $N(\langle M\rangle)=|\mathbb{P}^1(\mathbb{F}_q)|$ and therefore
 have $deg (\langle M\rangle)=|\mathbb{P}^1(\mathbb{F}_q)|=q+1 $. Since $M$ is a scalar
 matrix  $|t(\langle M\rangle)|= 1$.
 \item  By the statements in above three cases since $|N(\langle M\rangle)|=deg(\langle M\rangle)$, we have
    $|N(\langle M\rangle)|=1$ or $2$ or $|N(\langle M\rangle)|\geq 3$. Clearly only one of the above three cases hold.
 \end{enumerate}
\end{proof}
    Note that if $M\in E_2(\mathbb{F}_q)$, then its characteristic polynomial has
degree two. Hence the number of linearly independent eigendirections is
determined by the possible Jordan canonical forms of $M$ (see \cite[Corollary 4.7]{Hungerford}).
So we have only three stated cases in Lemma \ref{caseN(M)}.
For n=2 we summarize some properties of $\mathbf{SIG}(\mathbb{V}_0))$ in Table 1.
 Here, $dim E(M)$ denotes the  distributions of dimensions of the eigenspaces  of $M$.
\[\underset{Table1\quad \text{Possible cases for a matrix M where}\, n=2 \label{tabeln=3}}{
 {\footnotesize \begin{array}{|c|c|c|c|c|c|}
 \hline
\text{Row}&\chi_M(x)&\text{Jordan Canonical form }& dimE(M)&\deg(\langle M\rangle)&|t(\langle M\rangle)|\\
\hline
1&(x-\lambda)^2&J_2(\lambda)&(1)&1&q\\
\hline
2&(x-\lambda)^2&J_1(\lambda)\oplus J_1(\lambda)&(2)&q+1&1\\
\hline
3&(x-\lambda)(x-\mu)&J_1(\lambda)\oplus J_1(\mu), \lambda \neq\mu& (1,1)&2&q\\
\hline
\end{array}}}
\]
\begin{theorem}\label{TwinNontrivial}
Let $M\in E_n(\mathbb{F}_q)$ be a non-scalar matrix. Then
$t(\langle M\rangle)$ is a nontrivial twin class in $\mathbf{SIG}(\mathbb{V}_0)$.
More precisely, for every non-zero scalar
$c\in\mathbb F_q$, if $N=M+cI_n,$
then $t(\langle N\rangle)= t(\langle M\rangle)$.
\end{theorem}
\begin{proof}
Fix a non-zero $c\in\mathbb F_q$ and put
$N=M+cI_n.$ Since $M\in E_n(\mathbb F_q)$, there exists a non-zero vector
$v\in\mathbb V_0$ and a $\lambda\in\mathbb F_q$ such that
$Mv=\lambda v.$
Consequently, $
Nv=(M+cI_n)v=(\lambda+c)v$. Thus $v$ is also an eigenvector of $N$, and hence
$N\in E_n(\mathbb F_q)$. Moreover, for every $v\in\mathbb V_0\setminus\{0\}$, we have
\[
Mv=\lambda v
\quad\Longleftrightarrow\quad
(M+cI_n)v=(\lambda+c)v.
\]
Therefore $M$ and $N$ have exactly the same eigenvectors over
$\mathbb F_q$. In particular, they have the same projective
eigenspaces. By the definition of adjacency in
$\mathbf{SIG}(\mathbb V_0)$, it follows that $t(\langle N\rangle)= t(\langle M\rangle)$.
It remains to show that $\langle M\rangle$ and $\langle N\rangle$
are distinct matrix vertices. Suppose, to the contrary, that
$\langle N\rangle=\langle M\rangle$.
Then, by the definition of the projective matrix vertices, there
exists some $a\in\mathbb F_q$ such that
$N=aM$. Since $N=M+cI_n$, we obtain
$M+cI_n=aM$ and hence $cI_n=(a-1)M$.
Because $c\neq0$, this implies that $M$ is a scalar matrix,
contrary to the hypothesis. Thus
$\langle N\rangle\neq\langle M\rangle$.
Hence $\langle M\rangle$ and $\langle N\rangle$ are distinct
vertices having the same neighborhood, and consequently
$\langle M\rangle$ belongs to a nontrivial twin class.
\end{proof}

The following Theorem is about the twin classes of $\mathbf{SIG}(\mathbb{V}_0)$.
\begin{theorem}\label{twinpoints} For the graph $\mathbf{SIG}(\mathbb{V}_0)$ the following statements hold:
\begin{enumerate}
  \item  There is no nontrivial twin class in the part $\mathbb{P}^{n-1}(\mathbb{F}_q)$.
  \item  In part $\mathcal{E}^{n-1}(\mathbb{F}_q)$ there exists a positive integer $r$
  such that \[\mathcal{E}^{n-1}(\mathbb{F}_q)= \{\langle I_n\rangle\}\cup
   \mathcal{C}_2\cup \cdots \cup \mathcal{C}_r\] is the partition of the part
  $\mathcal{E}^{n-1}(\mathbb{F}_q)$ into twin classes, where every $\mathcal{C}_i$ is a nontrivial
   twin class.
 \end{enumerate}
\end{theorem}
\begin{proof}\indent (1) Let $\langle v_1\rangle$ and $\langle v_2\rangle$ be
arbitrary two distinct vertices in part $\mathbb{P}^{n-1}(\mathbb{F}_q)$.
 If they are twins then we have $|N(\langle v_1\rangle))\cup
  N(\langle v_2\rangle| = |N(\langle v_1\rangle)|=\frac{q^{n^2-n+1}-1}{q-1}$, by part $(1)$ of
 Theorem \ref{count}.
 On the other hand, by inclusion-exclusion Theorem and  Theorem \ref{intersection} we
  have:\\$|N(\langle v_1\rangle)\cup N(\langle v_2\rangle)| = |N(\langle v_1\rangle)|+
  |N(\langle v_2\rangle)|-|N(\langle v_1\rangle)\cap N(\langle v_2\rangle)|= 2\frac{q^{n^2-n+1}-1}{q-1}\\-\frac{q^{n^2-2n+2}-1}{q-1}=
\frac{q^{n^2-n+1}(2-q^{1-n})-1}{q-1}\neq\frac{q^{n^2-n+1}-1}{q-1}$, a contradiction.
So $\langle v_1\rangle$ and $\langle v_2\rangle$ are not twins.\\
\indent (2)
By Theorem \ref{TwinNontrivial} there exist a
  positive integer $r$ such that $\mathcal{E}^{n-1}(\mathbb{F}_q)=
   \{\langle I_n\rangle\}\cup \mathcal{C}_2\cup\cdots \cup \mathcal{C}_r$ is the decomposition
    of part $\mathcal{E}^{n-1}(\mathbb{F}_q)$ into twin classes, such that every $\mathcal{C}_i$
     is a nontrivial twin class.
\end{proof}
\section{\bf Automorphism group of $\mathbf{SIG}(\mathbb{V}_0$)}\label{section3}
The aim of this section is to determine the
  automorphism group of $\mathbf{SIG}(\mathbb{V}_0)$. We first
 introduce some notations and definitions and then establish  a sequence of lemmas and
theorems leading to the main results.\par
Note that a graph automorphism of the bipartite graph
$\mathbf{SIG}(\mathbb{V}_0)$ is a bijection
 \[\varphi :\mathcal{E}^{n-1}(\mathbb{F}_q )\cup \mathbb P^{n-1}(\mathbb{F}_q)\longrightarrow \mathcal{E}^{n-1}(\mathbb{F}_q) \cup \mathbb P^{n-1}(\mathbb{F}_q)\]
such that it preserves adjacency. Therefore,
 it must respect: \[\langle M\rangle\sim \langle v\rangle\Leftrightarrow
 \varphi(\langle M\rangle)\sim \varphi(\langle v\rangle)\quad \text{for all }
\langle M\rangle\in \mathcal{E}^{n-1}(\mathbb{F}_q )\,
 ,\ \langle v\rangle\in \mathbb{P}^{n-1}(\mathbb{F}_q).\]

 The following Lemma shows that every graph automorphism of $\mathbf{SIG}(\mathbb{V}_0)$
  preserves both of the parts of vertices of this graph. Indeed, every graph automorphism
  of $\mathbf{SIG}(\mathbb{V}_0)$ induces a permutation on the vertices of each part.
 \begin{lemma}\label{part to part} For every $\varphi \in Aut (\mathbf{SIG}(\mathbb{V}_0))$ we have
  $\varphi(\mathcal{E}^{n-1}(\mathbb{F}_q ))=\mathcal{E}^{n-1}(\mathbb{F}_q ) $ and
$\varphi (\mathbb P^{n-1}(\mathbb{F}_q))=\mathbb {P}^{n-1}(\mathbb{F}_q) $.
\end{lemma}
\begin{proof}
Using Theorem \ref{regular} for every $M \in E_n(\mathbb{F}_q)$  and every $0\neq v\in \mathbb{F}_q^n$
we have $deg(\langle M\rangle)< deg(\langle v\rangle)$. So
  $\varphi(\mathcal{E}^{n-1}(\mathbb{F}_q ))=\mathcal{E}^{n-1}(\mathbb{F}_q)$ and
$\varphi (\mathbb P^{n-1}(\mathbb{F}))=\mathbb P^{n-1}(\mathbb{F})$ as every graph automorphism preserves
 the degrees of vertices.
\end{proof}
\begin{remark}\label{identity} For  $0\neq r\in \mathbb{F}_q$ consider the map:
\[\rho_r:\mathcal{E}^{n-1}(\mathbb{F}_q)\cup \mathbb P^{n-1}(\mathbb{F}_q)\longrightarrow
\mathcal{E}^{n-1}(\mathbb{F}_q)\cup \mathbb P^{n-1}(\mathbb{F}_q)\] with
 $\rho_r(\langle M\rangle)=\langle rM\rangle=\langle M\rangle$ and
  $\rho_r(\langle v\rangle)=\langle rv\rangle =\langle v \rangle$ for every
$\langle M\rangle\in \mathcal{E}^{n-1}(\mathbb{F}_q )$ and
$\langle v\rangle\in \mathbb{P}^{n-1}(\mathbb{F}_q)$. So $\rho_r$
acts as identity graph automorphism of $\mathbf{SIG}(\mathbb{V}_0)$.
\end{remark}
\begin{lemma}\label{a5} Let $P\in GL_n(\mathbb{F}_q)$  and  let $\sigma$ be a field automorphism of
 $\mathbb{F}_q$.  Consider the map:
\[[P,\sigma]:\mathcal{E}^{n-1}(\mathbb{F}_q )\cup \mathbb P^{n-1}(\mathbb{F}_q)\longrightarrow
\mathcal{E}^{n-1}(\mathbb{F}_q)\cup \mathbb P^{n-1}(\mathbb{F}_q)\] with
 $[P,\sigma](\langle M\rangle)=\langle PM^\sigma P^{-1}\rangle$ and
  $[P,\sigma](\langle v\rangle)=\langle Pv^\sigma\rangle$ for every
$\langle M\rangle\in \mathcal{E}^{n-1}(\mathbb{F}_q )$ and
 $\langle v\rangle\in \mathbb{P}^{n-1}(\mathbb{F}_q)$, where $M^\sigma= (\sigma(m_{ij}))$ and
 $v^\sigma= (\sigma(v_i))$; that is, $\sigma$ acts entrywise.  Then $[P,\sigma]$
 is a graph automorphism of $\mathbf{SIG}(\mathbb{V}_0)$.
\end{lemma}
\begin{proof}If  $\langle M\rangle\sim \langle v\rangle$ then
 for a $\lambda \in \mathbb{F}_q$ we have
$ Mv=\lambda v$. So
 $\langle M\rangle\sim \langle v\rangle \Leftrightarrow \langle Mv\rangle=\langle\lambda v\rangle
\Leftrightarrow \langle PM^\sigma P^{-1} (Pv^\sigma )\rangle= \langle PM^\sigma v^\sigma\rangle=
\langle P(Mv)^\sigma\rangle=
\langle P(\lambda v)^\sigma\rangle=\langle\sigma(\lambda) Pv^\sigma\rangle= \langle Pv^\sigma\rangle$. Hence,
$\langle M\rangle\sim \langle v\rangle\Leftrightarrow [P,\sigma](\langle M\rangle)\sim
 [P,\sigma](\langle v\rangle)$
and this proves the Lemma.
\end{proof}
\begin{corollary}\label{fieldauto}  Let $\sigma \in Aut(\mathbb{F}_q)$ be  an arbitrary element in the group
of field automorphisms of $\mathbb{F}_q$. Consider the map:
\[\bar{\sigma}:\mathcal{E}^{n-1}(\mathbb{F}_q)\cup
 \mathbb P^{n-1}(\mathbb{F}_q)\longrightarrow \mathcal{E}^{n-1}(\mathbb{F}_q)\cup
  \mathbb P^{n-1}(\mathbb{F}_q)\] with $\bar{\sigma}(\langle M\rangle)=\langle M^\sigma\rangle$
  and $\bar{\sigma}(\langle v\rangle)=\langle v^\sigma\rangle$ for every
$M=(m_{ij})\in E_n(\mathbb{F}_q)$ and non-zero $v=(v_i)\in \mathbb{F}_q^n$,
 where $M^{\sigma}=(\sigma(m_{ij}))$ and $v^{\sigma}=(\sigma(v_i))$. Then $\bar{\sigma}$
 is a graph automorphism of $\mathbf{SIG}(\mathbb{V}_0)$.
\end{corollary}
\begin{proof} Let $P=I_n$ in Lemma \ref{a5}.
\end{proof}
Indeed, $\bar{\sigma}$ in Corollary \ref{fieldauto} is precisely $[I_n, \sigma]$ as defined in Lemma \ref{a5}.
 We call the graph automorphism in Corollary \ref{fieldauto},
 the \emph{graph automorphism induced by the field automorphism $\sigma$.}\par
Let \(V\) be a finite-dimensional vector space over a field $F$.
A bijection \[\varphi:\mathbb{P}(V)\longrightarrow\mathbb{P}(V)\]
is said to \emph{preserve collinearity} (i.e., $\varphi$ is a
\emph{collineation}; see \cite{Taylor1992}) if, for every three distinct projective points
(one-dimensional subspaces) $L_1,L_2,L_3\in\mathbb{P}(V),$
we have
\[
L_3\subseteq L_1+L_2
\quad\Longleftrightarrow\quad
\varphi(L_3)\subseteq
\varphi(L_1)+\varphi(L_2).
\]
\begin{theorem}\label{collinearity}
Let $\varphi$ be a graph automorphism of
$\mathbf{SIG}(\mathbb{V}_0)$. Then the restriction
$
\varphi|_{\mathbb{P}^{n-1}(\mathbb{F}_q)}
$
of $\varphi$ to the part $\mathbb{P}^{n-1}(\mathbb{F}_q)$ is a
collineation.
\end{theorem}
\begin{proof} By  Lemma \ref{part to part},
 $\varphi(\mathbb{P}^{n-1}(\mathbb{F}_q))=\mathbb{P}^{n-1}(\mathbb{F}_q)$.
  Let $L_1=\langle v_1\rangle, L_2=\langle v_2\rangle $ and
  $L_3=\langle v_3\rangle$  be three distinct
   projective points in $\mathbb{P}^{n-1}(\mathbb{F}_q)$, such that $L_3\subseteq L_1+ L_2$. Then,
$v_1, v_2,  v_3$ are pairwise linearly independent but the vectors $v_1, v_2,  v_3$ are
 linearly dependent.
So using  Lemma \ref{collinear} the number of matrices $M\in M_n(\mathbb{F}_q)$
 which has all three $v_i$s as eigenvectors is $q^{(n-1)^2}$. Note that every graph
 automorphism preserves degrees and neighborhoods and hence:
\begin{center}
$\varphi ( N(L_1)\cap N(L_2)\cap N(L_3))
=
N(\varphi(L_1))
\cap
N(\varphi(L_2))
\cap
N(\varphi(L_3)).$
\end{center}
Consequently, $\frac{q^{(n-1)^2-1}}{q-1}=|N(L_1)\cap N(L_2)\cap N(L_3)|
=|\varphi( N(L_1)\cap N(L_2)\cap$$N(L_3))|=
|N(\varphi(L_1))
\cap
N(\varphi(L_2))
\cap
N(\varphi(L_3))|$.

 Since $\varphi$ is an automorphism, $\varphi(L_1)$,
$\varphi(L_2)$ and
$\varphi(L_3)$ are three distinct projective points in $\mathbb{P}^{n-1}(\mathbb{F}_q)$ and
  clearly for each pair $1\leq i\neq j\leq3$ we have $dim (\varphi(L_i)+\varphi(L_j))=2$.
   We claim that  $dim(\varphi(L_1)+\varphi(L_2)+\varphi(L_3))=2$ and hence
$\varphi(L_3)\subseteq \varphi(L_1)+ \varphi(L_2)$.  Suppose, to the contrary, that
 we have $dim(\varphi(L_1)+\varphi(L_2)+\varphi(L_3))=3$.
 Let $\varphi(L_1)=\langle v'_1\rangle$, $\varphi(L_2)=\langle v'_2\rangle$
  and $\varphi(L_3)=\langle v'_3\rangle$ and hence the 
  vectors $v'_1, v'_2, v'_3 $ are linearly independent. So using
   Theorem \ref{intersection}
   \[|N(\varphi(L_1))\cap N(\varphi(L_2))\cap N(\varphi(L_3))|=\frac{q^{n^2-3n+3}-1}{q-1}.\]
 But for $n\geq 3$ we have $\frac{q^{n^2-3n+3}-1}{q-1}<\frac{q^{(n-1)^2}-1}{q-1}$, a contradiction. So
 $\varphi(L_3)\subseteq \varphi(L_1)+\varphi(L_2)$.
 Since $\varphi^{-1}$
 is also a graph automorphism, applying the same argument to $\varphi^{-1}$
shows that $\varphi(L_3)\subseteq \varphi(L_1)+\varphi(L_2)\Rightarrow L_3\subseteq L_1+L_2$. Therefore,
 $\varphi|_{\mathbb{P}^{n-1}(\mathbb{F}_q)}$ is a collineation,  for $n\geq 3$.
 For $n=2$ since every three vectors in a two-dimensional vector space
  are linearly dependent, the collinearity  is preserved by $\varphi|_{\mathbb{P}^{n-1}(\mathbb{F}_q)}$
   and this completes the proof.
\end{proof}

 $GL_n(\mathbb{F}_q)$, the general linear group of degree $n$, acts on the vector space
 $\mathbb{V}_0= \mathbb{F}_q^n$  by $v \mapsto Av$,
 for every $v \in \mathbb{F}_q^n$  and $A\in GL_n(\mathbb{F}_q)$.
In projective geometry, vectors differing by a non-zero scalar are identified
(i.e, for every non-zero $ v\in \mathbb{V}_0$, $v$ and $\lambda v$ are identified,
 for all $0\neq\lambda\in\mathbb{F}_q$).
  The set of $n\times n$ scalar matrices on $\mathbb{F}_q $ is
   $\{\lambda I_n: 0\neq\lambda\in \mathbb{F}_q\}$ and it is  a normal subgroup
     of $GL_n(\mathbb{F}_q)$.
Scalar matrices act trivially on $\mathbb{P}^{n-1}(\mathbb{F}_q)$, so we consider
 the quotient group: $PGL_n(\mathbb{F}_q) = GL_n(\mathbb{F}_q) / \{\lambda I_n\}$
  which is called the \emph{projective general linear group of degree} $n$
  over the field $\mathbb{F}_q$.
An element $A$ of $PGL_n(\mathbb{F}_q)$ acts on $\mathbb{P}^{n-1}(\mathbb{F}_q)$
  by $\langle v \rangle \mapsto \langle Av \rangle $.\\

 For a field $F$, let $V$ and $W$ be two vector spaces on $F$ and
$\sigma$ a field automorphism of $F$.
In linear algebra, particularly in projective geometry,
a map $T:V\to W$ is said to be a $\sigma$-\emph{semilinear transformation}
or simply \emph{semilinear transformation} if
$T(u+v)=T(u)+T(v)$ and  $T(\lambda v)=\sigma(\lambda)T(v)$
for all $\lambda\in F$ and $v\in V$.
If $T$ is a bijection then $T$ is said to be a $\sigma$-\emph{semilinear isomorphism}
 or simply \emph{semilinear isomorphism}.
If $V=W$ and $T$ is a bijection then $T$ is said to be a $\sigma$-\emph{semilinear automorphism}
  or simply \emph{semilinear automorphism} on $V$. If $V=W$, $\sigma=id_F$ and $T$ is a bijection
  then $T$ is said to be a \emph{linear automorphism} on $V$. (see \cite{Hirshfeld}).

\begin{remark}\label{semilinautoA} Every $\sigma$-semilinear automorphism
$T:\mathbb{F}^n_q\rightarrow \mathbb{F}^n_q$
acts as $T(v)=A \sigma(v)$, for a matrix $A\in GL_n(\mathbb{F}_q)$. To
see this, let $\{e_1,\dots,e_n\}$ be the standard basis of $\mathbb{F}^n_q$. For
$v=a_1e_1+\cdots+a_ne_n$, we have
$T(v)=T(a_1e_1)+\cdots+T(a_ne_n).$
Using semilinearity,
$T(a_i e_i)=\sigma(a_i)T(e_i)$. So $T(v)=\sigma(a_1)T(e_1)+\cdots+\sigma(a_n)T(e_n)$.
Define the matrix $A\in M_n(\mathbb{F}_q)$ with the $i$th column  $T(e_i)$ for each $1\leq i\leq n.$
Then
$T(v)=Av^\sigma$,
where $\sigma$ acts on each coordinate of $v$.
Since $T$ is bijective, the vectors
$T(e_1),\dots,T(e_n)$ are linearly independent, hence
$A\in GL_n(\mathbb{F}_q)$. We denote this $\sigma$-semilinear automorphism by $[A,\sigma]$.
\end{remark}

The symbol $\rtimes$ denotes the \emph{semidirect product} of two groups.
It is a standard construction in group theory. Let $G$ and $H$ be groups. We write $G\rtimes H$
when $H$ acts on $G$ by automorphisms. The action is given by a homomorphism
$\theta:$$H\rightarrow  Aut(G)$.
The elements of $G\rtimes H$ are pairs $[g,h]$ and multiplication is
$[g_1,h_1][g_2,h_2]=[g_1\theta(h_1)(g_2),h_1h_2])$.
Thus the symbol $\rtimes$ indicates that the second factor acts on the first factor.
When the action homomorphism is relevant, to
indicate the homomorphism $\theta$, it is denoted by $G\underset{\theta}\rtimes H$
(for more information see \cite[Section 5.5]{DF}).\\
In the following example there is a semidirect product of groups which is used in our next results.
\begin{example} \label{examplesemidirect}
 $ PGL_n(\mathbb{F}_q)\underset{\theta} \rtimes Aut(\mathbb{F}_q)$ has a semidirect product
 structure of groups, where $\theta:Aut(\mathbb{F}_q)\rightarrow Aut(PGL_n(\mathbb{F}_q)) $
      with $\sigma\mapsto \theta(\sigma)$ such that  $\theta(\sigma)(M)= M^{\sigma}$
  for every $M=( m_{ij})\in PGL_n(\mathbb{F}_q)$ and $M^\sigma=(\sigma(m_{ij}))$.  Then
  $\theta$ is a group homomorphism and
   for  $[A,\sigma]$ and $[B,\tau]$  we have
  $[A,\sigma][B,\tau]= [A\theta(\sigma)(B), \sigma\tau]= [AB^\sigma, \sigma\tau]$
   is the needed binary operation on $ PGL_n(\mathbb{F}_q) \rtimes Aut(\mathbb{F}_q)$
    and hence $PGL_n(\mathbb{F}_q)\underset{\theta}\rtimes Aut(\mathbb{F}_q)$
    is a semidirect product of groups.
\end{example}
\begin{remark}An element $[A,\sigma]$ of $ PGL_n(\mathbb{F}_q) \rtimes Aut(\mathbb{F}_q)$ also
 called a \emph{semilinear projectivity}, similarly to the Lemma \ref{a5}, $[A,\sigma]$ acts
  on  two parts of the vertices of $\mathbf{SIG}(\mathbb{V}_0)$, in particular on its
   projective part $\mathbb{P}^{n-1}(\mathbb{F}_q)$.
   Indeed a semilinear projectivity map on projective spaces
  is the projective version of a semilinear map on the vector spaces.
  \end{remark}
\begin{theorem}\label{semilinear}
Consider the group
$PGL_n(\mathbb{F}_q)\rtimes Aut(\mathbb{F}_q),$
whose elements are denoted by $[A,\sigma]$, where
$A\in PGL_n(\mathbb{F}_q)$ and
$\sigma\in Aut(\mathbb{F}_q)$. Let $A\in GL_n(\mathbb{F}_q)$
be a representative of the corresponding element of
$PGL_n(\mathbb{F}_q)$. The action of $[A,\sigma]$ on the projective
part $\mathbb{P}^{n-1}(\mathbb{F}_q)$ and on the matrix vertices part
$\mathcal{E}^{n-1}(\mathbb{F}_q)$ is defined by
$\langle v\rangle\longmapsto \langle Av^\sigma\rangle$
and
$\langle M\rangle\longmapsto
\langle AM^\sigma A^{-1}\rangle,$
respectively. These maps define a graph automorphism of
$\mathbf{SIG}(\mathbb{V}_0)$ for every $[A,\sigma]$. Consequently,
$Aut(\mathbf{SIG}(\mathbb{V}_0))$ contains a subgroup
isomorphic to
 $P\Gamma L_n(q)=PGL_n(\mathbb{F}_q)\rtimes Aut(\mathbb{F}_q).$
\end{theorem}

\begin{proof}
Let $[A,\sigma]\in
PGL_n(\mathbb{F}_q)\rtimes Aut(\mathbb{F}_q)$, where
$A\in GL_n(\mathbb{F}_q)$ is a representative of the corresponding
element of $PGL_n(\mathbb{F}_q)$. Then, similarly to the proof of Lemma \ref{a5}, $[A, \sigma]$
 is a graph automorphism of $\mathbf{SIG}(\mathbb{V}_0)$.
 Now, we show that the correspondence:
\[j: P\Gamma L_n(q)\rightarrow Aut(\mathbf{SIG}(\mathbb{V}_0))\]
\[ \text{ with   } j([A,\sigma])(\langle M\rangle)=\langle AM^\sigma A^{-1}\rangle,\;
j([A,\sigma])(\langle v\rangle)=\langle Av^\sigma\rangle
\]
gives a group embedding of
$P\Gamma L_n(q)$ into
$Aut(\mathbf{SIG}(\mathbb{V}_0))$. Let $[A, \sigma], [B, \alpha]$ be two elements of
 $PGL_n(\mathbb{F}_q) \rtimes Aut(\mathbb{F}_q)$, then as we saw in Example \ref{examplesemidirect},
  $j([A, \sigma][B, \alpha])= j( [AB^\sigma,\sigma\alpha])$. So for every vertex
   $\langle v\rangle\in \mathbb{P}^{n-1}(\mathbb{F}_q)$  and
    $\langle M\rangle \in \mathcal{E}^{n-1}(\mathbb{F}_q)$  we have:
 \[ j([A, \sigma][B, \alpha])(\langle v\rangle)= [AB^\sigma, \sigma\alpha](\langle v\rangle)=
  \langle AB^\sigma v^{\sigma\alpha}\rangle\] \[\text{and }\]
   $$  j([A, \sigma])j([B, \alpha])(\langle v\rangle)=
    j([A ,\sigma])[B,\alpha](\langle v\rangle)=
  j([A ,\sigma])\langle Bv^{\alpha}\rangle =\langle
    AB^\sigma v^{\sigma\alpha}\rangle $$
    Also we have:
  $ j([A, \sigma][B, \alpha])(\langle M\rangle)= [AB^\sigma, \sigma\alpha](\langle M\rangle)=
  \langle AB^\sigma M^{\sigma\alpha} (AB^\sigma)^{-1} \rangle$
  \[ \text{and}\]
 \[j([A, \sigma])j([B, \alpha])(\langle M\rangle)=
     j([A ,\sigma])\langle BM^{\alpha} B^{-1}\rangle = \langle
    A (B M^{\alpha}B^{-1})^{\sigma} A^{-1}\rangle
    =\]
    \[ \langle
    AB^\sigma M^{\sigma\alpha}(B^{-1})^\sigma A^{-1}\rangle
    =\langle AB^\sigma M^{\sigma\alpha}(B^\sigma)^{-1} A^{-1}\rangle=
    \langle AB^\sigma M^{\sigma\alpha}(AB^\sigma)^{-1}\rangle.\]
    Therefore, $j([A, \sigma][B, \alpha])=j([A, \sigma])j([B, \alpha])$
    and hence $j$ is a group homomorphism.
 To see that $j$ is injective, suppose that
$[A,\sigma]$ induces the identity automorphism of
$\mathbf{SIG}(\mathbb{V}_0)$. In particular, its restriction to the
projective part is the identity. Thus
$\langle Av^\sigma\rangle=\langle v\rangle$
for every nonzero $v\in\mathbb{F}_q^n$.
Let $\{e_1,\ldots,e_n\}$ be the standard basis of
$\mathbb{F}_q^n$.
Since $\sigma$ is a field automorphism of $\mathbb{F}_q$ we have
 $\sigma(0)=0$ and $\sigma(1)=1$ and hence $e_i^\sigma=e_i$ for every $i$. So from
 $\langle Av^\sigma\rangle=\langle v\rangle$
 we have
$\langle Ae_i\rangle=\langle e_i\rangle$
for every $i$. Therefore, there exist $a_1,\ldots,a_n\in\mathbb{F}_q\setminus \{0\}$
such that $Ae_i=a_i e_i$, for every $i$.
For $i\neq j$, consider the projective point
$\langle e_i+e_j\rangle$. Then, similarly we have $(e_i+e_j)^\sigma= e_i+e_j$.
Since the induced projective map is the
identity, we have
$\langle Ae_i+Ae_j\rangle
=
\langle e_i+e_j\rangle.
$
Hence
$
\langle a_i e_i+a_j e_j\rangle
=
\langle e_i+e_j\rangle,
$
which implies
$a_i=a_j.$
Therefore, there exists a non-zero $a\in\mathbb{F}_q$ such that
$A=aI_n.$
Now, for any $t\in\mathbb{F}_q$, consider the projective point
$\langle e_1+t e_2\rangle.$
Again, the induced projective map is the identity, and therefore
$
\langle Ae_1+\sigma(t)Ae_2\rangle
=
\langle e_1+t e_2\rangle.
$
Since $A=aI_n$, this gives
$
\langle e_1+\sigma(t)e_2\rangle
=
\langle e_1+t e_2\rangle.
$
Because the coefficient of $e_1$ is $1$ on both sides, it follows that
$\sigma(t)=t$
for every $t\in\mathbb{F}_q$. Hence
$\sigma=id_{\mathbb{F}_q}.$
Consequently,
$[A,\sigma]=[aI_n,id_{\mathbb{F}_q}]
=[I_n, id_{\mathbb{F}_q}]$
in $PGL_n(\mathbb{F}_q)\rtimes
Aut(\mathbb{F}_q)$. Thus $j$ is
injective. Therefore,
$P\Gamma L_n(q)$ is isomorphic to a subgroup of
$Aut(\mathbf{SIG}(\mathbb{V}_0)).$
\end{proof}
\begin{theorem}\label{FTPG}\emph{(}\textbf{\emph{Fundamental Theorem of Projective Geometry
\cite{Hirshfeld}}} \emph{)}\\
 \,\, Let $\mathbb{P}(V)$ and $\mathbb{P}(W)$ be projective spaces of dimension at least $2$,
arising from vector spaces $V$ and $W$ over fields $F$ and $K$, respectively.
Every projective collineation (a bijection mapping lines to lines) is
induced by a semilinear isomorphism $T:V\to W$. That is, if $f:\mathbb{P}(V)\to \mathbb{P}(W)$
is a collineation, then there exists a semilinear bijection $T$ such that
$f(\langle v\rangle)=\langle T(v)\rangle$
for every non-zero vector $v\in V$.
\end{theorem}
\begin{remark}Note that in  Theorem \ref{FTPG}, $n$ is the dimension of projective space not the
 dimension of  vector space. So the vector spaces in the Theorem \ref{FTPG} have the dimension
 at least 3.
 \end{remark}
 \begin{lemma}\label{corn3}
Let $n\geq 3$. Then for every
$\varphi\in Aut(\mathbf{SIG}(\mathbb{V}_0))$,
there exists a semilinear automorphism
$
T:\mathbb{F}_q^n\longrightarrow\mathbb{F}_q^n
$ such that
$
\varphi(\langle v\rangle)=\langle T(v)\rangle
$ for every $\langle v\rangle\in\mathbb{P}^{n-1}(\mathbb{F}_q)$.
Equivalently, there exist $A\in GL_n(\mathbb{F}_q)$ and
$\sigma\in Aut(\mathbb{F}_q)$ such that
$
\varphi(\langle v\rangle)
=
\langle Av^\sigma\rangle
$ for every $\langle v\rangle\in\mathbb{P}^{n-1}(\mathbb{F}_q)$.
Thus the restriction of $\varphi$ to
$\mathbb{P}^{n-1}(\mathbb{F}_q)$ is an element of
$
P\Gamma L_n(q)
=
PGL_n(\mathbb{F}_q)\rtimes Aut(\mathbb{F}_q).
$ Moreover, the restriction map
$
\Phi: Aut(\mathbf{SIG}(\mathbb{V}_0))
\longrightarrow
P\Gamma L_n(q),$
with $\Phi(\varphi)
=\varphi|_{\mathbb{P}^{n-1}(\mathbb{F}_q)},
$ is a group epimorphism.
\end{lemma}
\begin{proof}
By Theorem \ref{collinearity},
$\varphi|_{\mathbb{P}^{n-1}(\mathbb{F}_q)}
:
\mathbb{P}^{n-1}(\mathbb{F}_q)
\longrightarrow
\mathbb{P}^{n-1}(\mathbb{F}_q)
$
is a collineation of $\mathbb{P}^{n-1}(\mathbb{F}_q)$.
Since $n\geq3$, the projective space
$\mathbb{P}^{n-1}(\mathbb{F}_q)$ has dimension at least $2$.
Therefore, by the Fundamental Theorem of Projective Geometry
(Theorem \ref{FTPG}), there exists a semilinear automorphism
$T:\mathbb{F}_q^n\longrightarrow\mathbb{F}_q^n$
such that
$
\varphi(\langle v\rangle)=\langle T(v)\rangle
$ for every non-zero $v\in\mathbb{F}_q^n$.
Let $\sigma\in Aut(\mathbb{F}_q)$ be the associated
field automorphism of $T$. So by
Remark \ref{semilinautoA} there exists
$A\in GL_n(\mathbb{F}_q)$ such that
$T(v)=Av^\sigma$
for every $v\in\mathbb{F}_q^n$. Consequently,
$
\varphi(\langle v\rangle)
=
\langle Av^\sigma\rangle.
$
The projective transformation induced by $A$ depends only on the
class of $A$ in $PGL_n(\mathbb{F}_q)$. Hence
$
\varphi|_{\mathbb{P}^{n-1}(\mathbb{F}_q)}
\in
PGL_n(\mathbb{F}_q)\rtimes Aut(\mathbb{F}_q)
=
P\Gamma L_n(q).
$
We therefore obtain the well-defined restriction map
$
\Phi: Aut(\mathbf{SIG}(\mathbb{V}_0))
\longrightarrow
P\Gamma L_n(q),$ with
$
\Phi(\varphi)
=
\varphi|_{\mathbb{P}^{n-1}(\mathbb{F}_q)}.
$
It is a group homomorphism because, for
$\varphi,\psi\in Aut(\mathbf{SIG}(\mathbb{V}_0))$,
$$
\begin{aligned}
\Phi(\psi\circ\varphi)
&=
(\psi\circ\varphi)|_{\mathbb{P}^{n-1}(\mathbb{F}_q)}
&=
\psi|_{\mathbb{P}^{n-1}(\mathbb{F}_q)}
\circ
\varphi|_{\mathbb{P}^{n-1}(\mathbb{F}_q)}
&=
\Phi(\psi)\circ\Phi(\varphi).
\end{aligned}
$$
It remains to prove that $\Phi$ is surjective. Let
$[A,\sigma]\in P\Gamma L_n(q)$, where
$A\in$$GL_n(\mathbb{F}_q)$ and
$\sigma\in Aut(\mathbb{F}_q)$.  Define a map
$$
j([A,\sigma]):
\mathcal{E}^{n-1}(\mathbb{F}_q)
\cup
\mathbb{P}^{n-1}(\mathbb{F}_q)
\longrightarrow
\mathcal{E}^{n-1}(\mathbb{F}_q)
\cup
\mathbb{P}^{n-1}(\mathbb{F}_q)
$$
by $
j([A,\sigma])(\langle v\rangle)
=
\langle Av^\sigma\rangle
$ and $
j([A,\sigma])(\langle M\rangle)
=
\langle A M^\sigma A^{-1}\rangle.
$
By Theorem \ref{semilinear}, this map is an automorphism of
$\mathbf{SIG}(\mathbb{V}_0)$. Furthermore,
$
\Phi(j([A,\sigma]))
=
[A,\sigma].
$
Thus $\Phi$ is surjective. Hence $\Phi$ is a group epimorphism.
\end{proof}
   \begin{remark}\label{formproject}
   As we saw in the proof of Lemma \ref{corn3}, indeed the semidirect product
    $PGL_n(\mathbb{F}_q) \rtimes Aut(\mathbb{F}_q)$ is exactly the group of all
     semilinear projectivities of $\mathbb{P}^{n-1}(\mathbb{F}_q)$. So by Lemma \ref{corn3} every element of
    $PGL_n(\mathbb{F}_q) \rtimes Aut(\mathbb{F}_q)$ has the form $\varphi|_{\mathbb P^{n-1}(\mathbb{F}_q)}$
    for a $\varphi \in Aut(\mathbf{SIG}(\mathbb{V}_0))$
    and we use this form in the next Theorem.
   \end{remark}
 The assumption $n\geq3$ in Lemma \ref{corn3} is essential. When
$n=2$, the projective space is $\mathbb{P}^1(\mathbb{F}_q)$, and the
collinearity-preserving condition is not strong enough to force a
bijection of the projective points to be induced by a semilinear map.
Indeed, there are permutations of the projective points in
$\mathbb{P}^1(\mathbb{F}_q)$ that are not induced by semilinear maps.
As we will see in Lemmas \ref{extendS} and \ref{S_{q+1}} and
Theorem \ref{Autn=2}, such permutations can occur as the restrictions
of graph automorphisms of $\mathbf{SIG}(\mathbb{V}_0)$. In fact, when
$\mathbb{V}_0=\mathbb{F}_q^2$, every permutation of
$\mathbb{P}^1(\mathbb{F}_q)$ can be extended to an element of
$Aut(\mathbf{SIG}(\mathbb{V}_0))$, see Lemma \ref{extendS}.
\begin{theorem} \label{kernel} For $n\geq 3$ let $\mathcal{T}$ be the set of all twin classes
  in part $\mathcal{E}^{n-1}(\mathbb{F}_q)$. Consider the map:
{\footnotesize\[\Phi:Aut(\mathbf{SIG}(\mathbb{V}_0))\longrightarrow
PGL_n(\mathbb{F}_q) \rtimes Aut(\mathbb{F}_q),
\text{ with }\Phi(\varphi)= \varphi|_{\mathbb P^{n-1}(\mathbb {F}_q)}.\]}
Then $\Phi$ is a group epimorphism and the
 $ker(\Phi)\cong\prod_{\mathcal{C}\in \mathcal{T}}S_{\mathcal C}$
  is the direct product of  permutation groups on the twin classes in part
  $\mathcal{E}^{n-1}(\mathbb{F}_q)$.
\end{theorem}
\begin{proof}
   By Remark \ref{formproject} and Lemma \ref{corn3} $\Phi$ is a group epimorphism. We have:
\[ker(\Phi)=\{\varphi\in Aut(\textbf{SIG}(\mathbb{V})):
\varphi(\langle v\rangle)=\langle v\rangle
\text{ for all }\langle v\rangle\in \mathbb{P}^{n-1}(\mathbb F_q)\}.
\]
This is the subgroup of $Aut(\textbf{SIG}(\mathbb{V}_0))$ fixing every
 projective point $\langle v\rangle\in \mathbb{P}^{n-1}(\mathbb{F}_q)$ and
  $ker(\Phi) \trianglelefteq Aut(\mathbf{SIG}(\mathbb{V}_0))$.
If $\varphi\in ker(\Phi) $, then for every matrix vertex $\langle M\rangle$, we have
$N(\varphi(\langle M\rangle))=\varphi(N(\langle M\rangle))=N(\langle M\rangle)$.
Hence elements of $ker(\Phi)$ can only permute vertices having the same eigenvector set.
Consider the twin calss
\[
t(\langle M\rangle)=
\{\langle A\rangle \in \mathcal{E}^{n-1}(\mathbb{F}_q) :
N(\langle A\rangle)=N(\langle M\rangle)\}.
\]
Therefore $\mathcal{T}=\{t(\langle M\rangle): \langle M\rangle\in \mathcal{E}^{n-1}(\mathbb{F}_q)\}$.
  Thus, $ker(\Phi)$ is naturally identified with a
  subgroup of $\prod_{\mathcal{C}\in \mathcal{T}}S_{\mathcal C}$.
Conversely, let
$
\tau\in\prod_{\mathcal C\in\mathcal T}S_{\mathcal C}
$
be arbitrary. Define a permutation $\widehat{\tau}$ of the vertices
of $\mathbf{SIG}(\mathbb{V}_0)$ by fixing every vertex of
$\mathbb{P}^{n-1}(\mathbb{F}_q)$ and for each twin class
$\mathcal C\in\mathcal T$, letting $\widehat{\tau}$ act on
$\mathcal C$ according to the permutation $\tau$ on ${\mathcal C}$.
We claim that $\widehat{\tau}$ is a graph automorphism. Indeed, if
$\langle M\rangle$ and $\langle M'\rangle$ belong to the same twin
class, then
$
N(\langle M\rangle)=N(\langle M'\rangle).
$
Hence
$
N(\widehat{\tau}(\langle M\rangle))
=
N(\langle M\rangle).
$
Since every projective vertex is fixed, adjacency is preserved, as
for every matrix vertex $\langle M\rangle$ and every projective
vertex $\langle v\rangle$,
$$
\langle M\rangle\sim\langle v\rangle
\iff
\langle v\rangle\in N(\langle M\rangle)
\iff
\langle v\rangle\in
N(\widehat{\tau}(\langle M\rangle))
\iff
\widehat{\tau}(\langle M\rangle)\sim\langle v\rangle.
$$
Thus $\widehat{\tau}\in Aut(\mathbf{SIG}(\mathbb{V}_0))$.
Moreover, $\widehat{\tau}$ fixes every projective point, so
$
\widehat{\tau}\in ker(\Phi).
$
Therefore every element of
$\prod_{\mathcal C\in\mathcal T}S_{\mathcal C}$ occurs in
$ker(\Phi)$, and hence
$
ker(\Phi)
=
\prod_{\mathcal C\in\mathcal T}S_{\mathcal C}.
$
This proves the result.
\end{proof}

\begin{remark}\label{T_i}  Let $\mathcal{T}$ be the set of all twin classes
  in the part $\mathcal{E}^{n-1}(\mathbb{F}_q)$. Then by Theorem \ref{twinpoints}  there exists an integer
  $m\in \mathbb{N}$  such that $\mathcal{T}=\mathcal{T}_1\cup \cdots \cup \mathcal{T}_m$
is a partition of $\mathcal{T}$, such that each $\mathcal{T}_i$ is the set of all twin classes
 of matrix vertices with the same degree( i.e., if  $A, B \in \mathcal{T}_i$
  are  two arbitrary  twin classes, then for every two vertices
   $\langle M\rangle\in A$ and $\langle M'\rangle\in B$
    we have $deg (\langle M\rangle)=deg(\langle M'\rangle)$ ).
 \end{remark}
For $n\geq 3$, the following theorem  determines the  automorphism
group of the graph $\mathbf{SIG}(\mathbb{V}_0)$. For $n=2$ it is done in Theorem \ref{Autn=2}.

\begin{theorem} \label{Autn-geq3}
Assume $n\geq 3$. Let $\mathcal{T}$ be the set of all twin classes
  in the part $\mathcal{E}^{n-1}(\mathbb{F}_q)$ and $\Phi$ the epimorphism stated in
 Theorem \emph{\ref{kernel}}.
 Then the sequence:
{\footnotesize\[
1\longrightarrow
\prod_{\mathcal C\in\mathcal T}
S_{\mathcal C}\overset{i}
\longrightarrow
Aut(\mathbf{SIG}(\mathbb V_0))
\overset{\Phi}{\longrightarrow}
PGL_n(\mathbb{F}_q) \rtimes Aut(\mathbb{F}_q)
\longrightarrow 1,
\]}
is a split short exact sequence of groups, where $i$ is the inclusion homomorphism. So
{\footnotesize\[
Aut(\mathbf{SIG}(\mathbb V_0))
\cong
\prod_{\mathcal C\in\mathcal T}
S_{\mathcal C}\underset{\theta}\rtimes P\Gamma L_n(q),
\]}
where $ P\Gamma L_n(q)= PGL_n(\mathbb F_q)\rtimes Aut(\mathbb F_q)$
and the action homomorphism
{\footnotesize \[
\theta:
P\Gamma L_n(q)\rightarrow Aut(\prod_{\mathcal C\in\mathcal T}
S_{\mathcal C})
\]} is given by the conjugation action :
{\footnotesize\[\theta([M,\alpha])(\tau)
=[M,\alpha]\tau[M,\alpha]^{-1} \text{ for every } \tau \in \prod_{\mathcal C\in\mathcal T}
S_{\mathcal C}.\]}
\end{theorem}
\begin{proof}
By Theorem \ref{kernel}, $\Phi$ is surjective with kernel $\underset{\mathcal C\in\mathcal T}\prod
S_{\mathcal C}$. So the sequence is exact. By Theorem \ref{semilinear}
  $PGL_n(\mathbb{F}_q) \rtimes Aut(\mathbb{F}_q)$ is isomorphic to a subgroup of
  $Aut(\mathbf{SIG}(\mathbb{V}_0))$,  with the inclusion $j$, where
  $j([A,\sigma])(\langle M\rangle)=\langle AM^\sigma A^{-1}\rangle$ and
  $j([A,\sigma])(\langle v\rangle)=\langle Av^\sigma\rangle.$
   It was shown in Theorem \ref{semilinear} that $j([A, \sigma])\in Aut(\mathbf{SIG}(\mathbb{V}_0))$
 and $j$ is a group monomorphism.
    Note that $\Phi \circ j([A,\sigma])= \Phi([ A,\sigma])
     =[A,\sigma]$, for every $[A,\sigma]\in P\Gamma L_n(q)$.  So
   $ \Phi\circ j = id_{ P\Gamma L_n(q)}$
  and hence the sequence splits( see \cite[Definition in pages 383 and 384]{DF}).
 Therefore,{\small \[Aut(\mathbf{SIG}(\mathbb V_0))
\cong
\prod_{\mathcal C\in\mathcal T}
S_{\mathcal C}\underset{\theta}\rtimes P\Gamma L_n(q),\]}
and the action homomorphism
\[
\theta:
P\Gamma L_n(q)\rightarrow Aut(\prod_{\mathcal C\in\mathcal T}
S_{\mathcal C})\]
is given by the conjugation action $\theta([M,\alpha])(\tau)
=[M,\alpha]\tau[M,\alpha]^{-1}\in ker(\Phi)$, as $ker(\Phi)$ is a normal subgroup
 of $Aut(\mathbf{SIG}(\mathbb{V}_0))$ and this completes the proof.
\end{proof}
\begin{theorem}\label{extendS} Assume that $n=2$.
For every subset $X\subseteq \mathbb{P}^{1}(\mathbb{F}_q)$ with $1\leq$$|X|\leq $$2$,
 let $A_X= \{\langle M\rangle\in \mathcal{E}^1(\mathbb{F}_q): N(\langle M\rangle)= X\}$.
  By Lemma \emph{\ref{caseN(M)}},  $|A_X|=q$ for each such $X$.
For every such $X$, fix a labeling
$A_X=\{\langle M_{X,1}\rangle, \langle M_{X,2}\rangle,\ldots, \langle M_{X,q}\rangle\}$.
Then, for every permutation $\sigma\in S_{\mathbb{P}^{1}(\mathbb{F}_q)}$,
 there exists a permutation
$\widehat{\sigma}\in $$Aut(\mathbf{SIG}(\mathbb V_0))$ satisfying
$\widehat{\sigma}(\langle v\rangle)=\sigma(\langle v\rangle)$ for every
 $\langle v\rangle\in \mathbb{P}^{1}(\mathbb{F}_q) )$ and
$\widehat{\sigma}(\langle M_{X,i}\rangle)=\langle M_{\sigma(X),i}\rangle$
for every $1\leq i\leq q$, where
$\sigma(X)=\{\sigma(\langle v\rangle):\langle v\rangle\in  X\}$, and
$\widehat{\sigma}(\langle I_2\rangle)=\langle I_2\rangle$. Moreover, the map
$\Psi: S_{\mathbb{P}^{1}(\mathbb{F}_q)}\longrightarrow Aut(\mathbf{SIG}(\mathbb V_0))$
with $\Psi(\sigma)=\widehat{\sigma}$, is an injective group homomorphism. Consequently,
$S_{\mathbb{P}^{1}(\mathbb{F}_q)}\cong S_{q+1}$ is isomorphic to a subgroup of
$Aut(\mathbf{SIG}(\mathbb V_0))$.
\end{theorem}
\begin{proof}
Note that by Lemma~\ref{caseN(M)}, every non-scalar matrix vertex has a neighborhood of
cardinality 1 or 2, and hence the set $$\{A_X: X\subseteq \mathbb{P}^{1}(\mathbb{F}_q)\text{ and }
1\le |X|\le 2\}$$ is a partition of the non-scalar matrix vertices into non-trivial twin classes.
Moreover, $A_X\cap A_Y=\varnothing$  whenever $X\ne Y,$ for every such $X$ and $Y$.\par
Now, let $\sigma\in S_{\mathbb{P}^{1}(\mathbb{F}_q)}$.
 We define a permutation $\widehat{\sigma}$ of the vertices of $\mathbf{SIG}(\mathbb V_0)$ as follows.
First, on the part $\mathbb{P}^{1}(\mathbb{F}_q)$, define
$\widehat{\sigma}(\langle v\rangle)=\sigma(\langle v\rangle)$,
for every $ \langle v\rangle\in \mathbb{P}^{1}(\mathbb{F}_q)$.
For every subset $X\subseteq \mathbb{P}^{1}(\mathbb{F}_q)$ with $1\leq |X|\leq2$ and every
$1\leq i\leq q$, define
$\widehat{\sigma}(\langle M_{X,i}\rangle) = \langle M_{\sigma(X),i}\rangle$.
Finally, fix the unique singleton twin class consisting of the scalar matrix vertex:
$\widehat{\sigma}(\langle I_2\rangle) = \langle I_2\rangle$.
Since $\sigma$ permutes the subsets of $\mathbb{P}^{1}(\mathbb{F}_q)$ of each cardinality of 1 and 2,
$\widehat{\sigma}$ is a bijection on the vertex set of the graph.
We show that $\widehat{\sigma}$ preserves adjacency.
 The scalar vertex $\langle I_2\rangle$ is adjacent to every projective vertex, and hence
  fixing it preserves all its incident edges.
Now consider a non-scalar matrix vertex $\langle M_{X,i}\rangle$ and
a projective vertex $\langle v\rangle\in \mathbb{P}^{1}(\mathbb{F}_q) $.
By the definition of the twin class $A_X$,
$ \langle M_{X,i}\rangle \sim \langle v\rangle \Longleftrightarrow \langle v\rangle\in X$.
Therefore,
\[\begin{aligned} \langle M_{X,i}\rangle \sim \langle v\rangle &\Longleftrightarrow \langle v\rangle\in X\\
&\Longleftrightarrow \sigma(\langle v\rangle)\in\sigma(X)\\
&\Longleftrightarrow \langle M_{\sigma(X),i}\rangle \sim\sigma(\langle v\rangle)\\
&\Longleftrightarrow \widehat{\sigma}(\langle M_{X,i}\rangle) \sim \widehat{\sigma}(\langle v\rangle).
 \end{aligned}\]
 Thus $\widehat{\sigma}$ preserves adjacency and is consequently an automorphism of
 $\mathbf{SIG}(\mathbb V_0)$.
 Let $\sigma,\tau\in S_{\mathbb{P}^{1}(\mathbb{F}_q)}$. For every
  $\langle v\rangle\in \mathbb{P}^{1}(\mathbb{F}_q)$,
 \[\widehat{\sigma}\widehat{\tau}(\langle v\rangle) =
  \widehat{\sigma}(\tau(\langle v\rangle)) =
  \sigma(\tau(\langle v\rangle))=
   (\sigma\tau)(\langle v\rangle)=
   \widehat{\sigma\tau}(\langle v\rangle).\]
   For every matrix vertex $\langle M_{X,i}\rangle$,
\[\begin{aligned} \widehat{\sigma}\widehat{\tau}(\langle M_{X,i}\rangle) =
\widehat{\sigma}(\langle M_{\tau(X),i}\rangle)
 = \langle M_{\sigma(\tau(X)),i}\rangle
 = \langle M_{(\sigma\tau)(X),i}\rangle
 = \widehat{\sigma\tau}(\langle M_{X,i}\rangle).
  \end{aligned}\]
Both maps also fix $\langle I_2\rangle$. Hence
$\widehat{\sigma}\widehat{\tau} = \widehat{\sigma\tau}$.
Thus $\Psi$ is a group homomorphism.
Finally, if $\Psi(\sigma)=\Psi(\tau)$, then their restrictions to $\mathbb{P}^{1}(\mathbb{F}_q)$ are equal.
 Hence $\sigma(\langle v\rangle)=\tau(\langle v\rangle)\text{ for every }\langle v\rangle\in
  \mathbb{P}^{1}(\mathbb{F}_q)$,
 and therefore $\sigma=\tau$. Thus $\Psi$ is injective. Since $|\mathbb{P}^{1}(\mathbb{F}_q)|=q+1$, we have
$S_{\mathbb{P}^{1}(\mathbb{F}_q)}\cong S_{q+1}$, which proves the result.
\end{proof}
\begin{lemma}\label{S_{q+1}} Let $n=2$. Then the map
$\Phi:Aut(\mathbf{SIG}(\mathbb V_0))\longrightarrow
S_{\mathbb P^{1}(\mathbb F_q)},$
 with $\Phi(\varphi)= \varphi|_{\mathbb P^{1}(\mathbb {F}_q)}$ is a group epimorphism and
 $\frac{Aut(\mathbf{SIG}(\mathbb{V}_0))}{\prod_{\mathcal{C}\in
  \mathcal{T}}S_{\mathcal C}}\cong S_{q+1} $, where $\mathcal{T}$ is the set of
all twin classes in the part $\mathcal{E}^1(\mathbb{F}_q)$.
\end{lemma}
\begin{proof}
Using Lemma \ref{part to part} and Theorem \ref{extendS} clearly $\Phi$  is surjective
and since $|\mathbb{P}^1(\mathbb{F}_q)|=q+1$, we
have $S_{\mathbb{P}^1(\mathbb{F}_q)}\cong S_{q+1}$. Now,
similarly to the proof of Theorem \ref{kernel} we have
$ker(\Phi)=\prod_{\mathcal{C}\in \mathcal{T}}S_{\mathcal C}$ and 
 $Im(\Phi)= S_{\mathbb{P}^1(\mathbb{F}_q)}\cong S_{q+1}$, so
by the First Isomorphism Theorem for groups
 we have: \[\frac{Aut(\mathbf{SIG}(\mathbb V_0))}{\prod_{\mathcal{C}\in
  \mathcal{T}}S_{\mathcal C}}\cong S_{\mathbb{P}^1(\mathbb{F}_q)}\cong S_{q+1}.\]
\end{proof}
 \begin{theorem}\label{Autn=2}
Let $n=2$. Then there exists a split short exact sequence
{\footnotesize\[
1\longrightarrow
\left(
\prod_{\mathcal C\in\mathcal T_1}
S_{\mathcal C}
\times
\prod_{\mathcal C\in\mathcal T_2}
S_{\mathcal C}
\right)
\overset{i}\longrightarrow
Aut(\mathbf{SIG}(\mathbb V_0))
\overset{\Phi}{\longrightarrow}
S_{\mathbb P^1(\mathbb F_q)}
\longrightarrow 1,
\]}
where $\mathcal T_1$ is the set of all twin classes in the part
$\mathcal E^1(\mathbb F_q)$ whose vertices have exactly one
eigendirection with $|\mathcal{T}_1|=q+1$ and $\mathcal T_2$ is the set of all twin classes in the part
$\mathcal E^1(\mathbb F_q)$ whose
vertices have exactly two distinct eigendirections with $|\mathcal{T}_2|=\frac{q(q+1)}{2}$. Consequently,
{\footnotesize\[Aut(\mathbf{SIG}(\mathbb V_0))
\cong
\left(
\prod_{\mathcal C\in\mathcal T_1}
S_q
\times
\prod_{\mathcal C\in\mathcal T_2}
S_q
\right)
\underset{\theta}\rtimes S_{q+1}
=\left(
\prod_{i=1}^{q+1}S_q
\times
\prod_{i=1}^{\frac{q(q+1)}{2}}S_q
\right)
\underset{\theta}\rtimes S_{q+1}
,\]}
 where  the action defining the semidirect product is
{\footnotesize \[ \theta:S_{q+1}\cong S_{\mathbb P^1(\mathbb F_q)}\longrightarrow Aut\left(
\prod_{\mathcal C\in\mathcal T_1}S_q
\times
\prod_{\mathcal C\in\mathcal T_2}S_q
\right).\]}
is given by  conjugation: $\theta(\sigma)(\kappa)=\widehat{\sigma}\kappa{\widehat{\sigma}}^{-1}$ and
 $\widehat{\sigma}$ is an extension  of $\sigma$ to a graph automorphism of $\mathbf{SIG}(\mathbb{V}_0)$ as
 in Theorem \ref{extendS}.
\end{theorem}
\begin{proof}
Let $
\Phi:
Aut(\mathbf{SIG}(\mathbb V_0))
\longrightarrow
S_{\mathbb P^1(\mathbb F_q)}
$
be the restriction homomorphism defined by
$\Phi(\varphi)
=\varphi|_{\mathbb P^1(\mathbb F_q)}.$
By Lemma \ref{S_{q+1}}, $Im(\Phi)= S_{\mathbb P^1(\mathbb F_q)}
\cong S_{q+1},$ hence $\Phi$ is an epimorphism.
Also by Lemma \ref{S_{q+1}}, $ker(\Phi)
=\prod_{\mathcal C\in\mathcal T}
S_{\mathcal C},$
where $\mathcal T$ is the set of all twin classes in the part
$\mathcal E^1(\mathbb F_q)$.
By Lemma \ref{caseN(M)}, we have two possible types of  nontrivial twin classes of matrix vertices.\\
\medskip
\noindent
\textbf{Case 1:}
$|N(\langle M\rangle)|=1$. In this case $M$ has exactly one eigendirection. By Lemma \ref{caseN(M)},
for every $\langle v\rangle\in\mathbb P^1(\mathbb F_q)$
there are exactly $q$ matrix vertices having neighborhood $\{\langle v\rangle\}$
and hence these matrix vertices make a twin class
 of cardinality $q$.  Thus $|\mathcal{T}_1|=\begin{pmatrix}
                     |\mathbb P^1(\mathbb F_q)|\\
                     1
         \end{pmatrix}= \begin{pmatrix}
                     q+1\\
                     1
         \end{pmatrix}=q+1$.
 Note that for every twin class $\mathcal{C}$ in $\mathcal T_1$
   we have $|\mathcal C|=q$ and hence $S_{\mathcal C}\cong S_q.$\\
\noindent
\textbf{Case 2:}
$|N(\langle M\rangle)|=2$.
In this case $M$ has two distinct eigenvalues and two eigendirections.
By Lemma \ref{caseN(M)}
for each pair $\{\langle u\rangle,\langle v\rangle\}$
of distinct points of $\mathbb P^1(\mathbb F_q)$,
 there exist exactly
$q$ matrix vertices having the neighborhood $\{\langle u\rangle,\langle v\rangle\}$. Thus $|\mathcal{T}_2|=\begin{pmatrix}
                     q+1\\
                     2
         \end{pmatrix}=\frac{(q+1)!}{2!(q-1)!}=\frac{q(q+1)}{2}$.  Since every
  twin class $\mathcal{C}$ in $\mathcal T_2$ has cardinality $q$, we have
$S_{\mathcal C}\cong S_q.$  The scalar matrix class gives only the single vertex
$\langle I_2\rangle$ and contributes no nontrivial factor.
Thus
{\footnotesize \[
ker(\Phi)
\cong
\left(
\prod_{\mathcal C\in\mathcal T_1}S_q
\times
\prod_{\mathcal C\in\mathcal T_2}S_q
\right)=\left(
\prod_{i=1}^{q+1}S_q\times \prod_{i=1}^{\frac{q(q+1)}{2}}S_q
\right).\]}
Consequently we have the short exact sequence of groups:
\[
1\longrightarrow
ker(\Phi)
\overset{i}\longrightarrow
Aut(\mathbf{SIG}(\mathbb V_0))
\overset{\Phi}{\longrightarrow}
S_{\mathbb P^1(\mathbb F_q)}
\longrightarrow1 .
\]
It remains to show that this sequence splits.
By Theorem \ref{extendS}, the group
$S_{\mathbb P^1(\mathbb F_q)}$  is isomorphic to a subgroup of
$ Aut(\mathbf{SIG}(\mathbb V_0))$.
By Theorem \ref{extendS}
the map
\[\Psi:S_{\mathbb{P}^1(\mathbb{F}_q)}\cong S_{q+1}\longrightarrow Aut(\mathbf{SIG}(\mathbb{V}_0))\] is a
group monomorphism and it is easy to see  that $\Phi\circ \Psi= id_{S_{\mathbb P^1(\mathbb F_q)}}.$
Hence the short exact sequence of groups splits and therefore,
{\footnotesize\[Aut(\mathbf{SIG}(\mathbb V_0))
\cong
ker(\Phi)\rtimes S_{\mathbb P^1(\mathbb F_q)}
\cong\left(
\prod_{i=1}^{q+1}S_q\times \prod_{i=1}^{\frac{q(q+1)}{2}}S_q
\right)
\underset{\theta}\rtimes S_{q+1}.
\]}
 The action defining the semidirect product is
 {\footnotesize\[ \theta:S_{q+1}\cong S_{\mathbb{P}^1(\mathbb{F}_q)}\longrightarrow \operatorname{Aut}\left(
\prod_{\mathcal C\in\mathcal T_1}S_q
\times
\prod_{\mathcal C\in\mathcal T_2}S_q
\right).\]}
Note that
$ker(\Phi)\cong
\left(
\prod_{\mathcal C\in\mathcal T_1}S_q
\times
\prod_{\mathcal C\in\mathcal T_2}S_q
\right)$ is a normal subgroup of the group $Aut(\mathbf{SIG}(\mathbb{V}_0))$,
 so by conjugation we have $\theta(\sigma)(\kappa)=\widehat{\sigma}\kappa\widehat{\sigma}^{-1}\in ker(\Phi),$
  for all $k\in ker(\Phi)$ and  hence  $\theta(\sigma)\in Aut(ker(\Phi))$,
where $\widehat{\sigma}$ is the extension of $\sigma$  to $Aut(\mathbf{SIG}(\mathbb{V}_0))$
as in Theorem \ref{extendS}.
\end{proof}
\section{Declarations}
\textbf{Conflict of interest:} The author has no relevant financial or non-financial interests to disclose.\\
\bibliographystyle{elsarticle-num}
 
 \end{document}